\documentclass[a4paper]{article}
\usepackage[utf8]{inputenc} 
\usepackage{hyperref}
\usepackage{amsfonts,amsmath,amstext,amssymb}
\usepackage{graphicx}
\usepackage{hyperref}\hypersetup{colorlinks=true, linkcolor=blue, citecolor=blue, filecolor=blue, urlcolor=red, pdfauthor={Nome do Autor}, pdftitle={On two problems in quaternionic hyperbolic geometry},          pdfsubject={Assunto do Projeto}, pdfkeywords={palavra-chave, palavra-chave, palavra-chave}, pdfproducer={Latex},        pdfcreator={pdflatex}}
\def\mod#1{\vert #1\vert}

\def\A{\mathbb A}
\def\B{\mathbb B}

\def\R{\mathbb R}
\def\C{\mathbb C}

\def\Z{\mathbb Z}
\def\Q{\mathbb Q}
\def\P{\mathbb P}

\def\F{\mathbb F}

\def\herm#1{\langle #1\rangle}

\newenvironment{proof}{\noindent\normalsize {\sc Proof}:}{{\hfill \rule{1mm}{3mm}}}

\newtheorem{theorem}{Theorem}[section]
\newtheorem{co}{Corollary}[section]
\newtheorem{prop}{Proposition}[section]
\newtheorem{lemma}{Lemma}[section]
\newtheorem{rmk}{Remark}[section]

\title{Moduli of triples of points in quaternionic hyperbolic geometry}

\author{Igor Almeida\ \ \ \ \ \ \ \ \ \ Nikolay Gusevskii\thanks{Corresponding author.}\\
almeida@mat.ufmg.br\ \ \ \ nikolay@mat.ufmg.br\\
\\
Departamento de Matem\'{a}tica\\
Universidade Federal de Minhas Gerais \\
Belo Horizonte -- MG\\
Brazil\\
30123-970}

\date{}

\begin{document}
\maketitle

%----------------------------------------------------------------------------------

\begin{abstract}
\noindent In this work, we describe the moduli  of triples of points
in quaternionic projective space which define uniquely the congruence classes of such triples relative to the action of the isometry group of quaternionic hypebolic space ${\rm H^{n}_{\Q}}$.
To solve this problem, we introduce some basic invariants of triples of points in quaternionic hyperbolic geometry.
In particular, we define  quaternionic analogues of the Goldman invariants for mixed configurations of points introduced by him in complex
hyperbolic geometry.
\end{abstract}



\section*{Introduction}

\noindent The  purpose of this paper is to describe some numerical invariants associated to an ordered triple of points in quaternionic projective space.
These invariants describe the equivalence classes of such triples relative to the action of the isometry group of quaternionic hypebolic space ${\rm H^{n}_{\Q}}$.
We  give a construction of the quaternionic angular invariant, an analogue of the Cartan invariant in complex hyperbolic geometry, see \cite{Car}, which parametrizes triples of isotropic points.
Also, we represent a quaternionic analogue of Brehm's shape invariant, see  \cite{Bre}, in complex hyperbolic geometry, which is used to parametrize triples of points in ${\rm H^{n}_{\Q}}$.
Then we define  a quaternionic analogue of the Goldman $\eta$-invariant for mixed configurations of points introduced by him in complex
hyperbolic geometry to study the intersection of bisectors, see \cite{Gol}. Using these invariants, we describe the moduli of the corresponding triples
relative to the action of the isometry group of quaternionic hypebolic space ${\rm H^{n}_{\Q}}$. In order to solve the congruence problems, we use the methods related to Gram matrices of configurations of points
developed in complex hyperbolic geometry in \cite{AGG}, \cite{Bre},  \cite{BrEt},  \cite{CDGT},  \cite{CG1},  \cite{CG2},  \cite{HS1}, \cite{HS2}. In this work, we describe the moduli of all possible configurations of three
points in quaternionic projective space of any dimension and give a geometric interpretation of them.

We remark that some of these problems were considered by Cao, see \cite{Cao}.
Unfortunately, some of the main results of this work are not correct as stated, see Theorem 1.1  (items (ii) and (iii)) in  \cite{Cao}.  We provide the corresponding counter-examples, see Section 2.4.1 and Section 2.4.2. It will be shown that in order to describe the moduli of the corresponding triples we need a quaternionic analogue of the Goldman $\eta$-invariant
for mixed configurations of points.

\medskip

The work is organized as follows. In Section 1,  we summarize some basic results about geometry of quaternionic hyperbolic space. In Section 2, we describe
the moduli of triples of points in quaternionic hyperbolic geometry.

\section{Preliminaries}

\noindent In this section, we recall some basic results  related to quarternions and geometry of projective and hyperbolic spaces.

\subsection{Quaternions}

\noindent First, we recall some basic facts about the quaternions we need. The quaternions $\Q$ are the $\R$-algebra generated by the symbols $i,j,k$ with the relations

$$ l^2=j^2=k^2=-1,  \ \ \  ij=-ji=k, \ \ \  jk=-kj=1, \ \ \  ki=-ik=j.$$

\medskip

So, $\Q$ is a skew field and a 4-dimensional division algebra over the reals.

\medskip

Let $a \in \Q$. We write $a=a_0 + a_1 i + a_2 j + a_3 k$, $a_i \in \R$,
then by definition

$$\bar{a}= a_0 - a_1 i - a_2 j - a_3 k, \ \ \ {\rm Re} \ a=a_0, \ \ \ {\rm Im} \ a=a_1 i + a_2 j + a_3 k.$$

\medskip

Note that, in contrast with the complex numbers, ${\rm Im} \ a$ is not a real number (if $a_i\neq 0$ for some $i=1,2,3$), and that conjugation obeys the rule

$$\overline{ab} = \bar{b}\bar{a}.$$

\medskip

Also, we define $\mod{a}=\sqrt{a\bar{a}}$. We have that if $a\neq 0$ then $a^{-1}=\bar{a}/\mod{a}^2.$

\medskip

In what follows, we will identify the reals  numbers $\R$ with ${\R} 1$ and the complex numbers $\C$ with the subfield of $\Q$ generated over $\R$ by $1$ and $i$.

\medskip

Two quaternions $a$ and $b$ are called {\sl similar} if there exists $\lambda \neq 0$ such that $a=\lambda b \lambda^{-1}.$
By replacing $\lambda$ by $\tau=\lambda/\mod{\lambda}$, we may always assume $\lambda$ to be unitary.

\medskip

The following proposition was proved in \cite{Bren}.

\begin{prop}\label{qsim} Two quaternions $a$ and $b$ are similar if and only if ${\rm Re} \ a ={\rm Re} \ b$ and
$\mod{a} = \mod{b}$. Moreover, every similarity class contains a complex number, unique up to conjugation.
\end{prop}

\begin{co} \label{qua1} Any quaternion $a$ is similar to a unique complex number $b=b_0 +b_1 i $ such that $b_1 \geq 0.$
\end{co}

Also, this proposition implies that every quaternion is similar to its conjugate.

\medskip

We say that $a\in \Q$ is {\sl purely imaginary} if ${\rm Re}(a)=0 $. Let us suppose that $a$ is purely imaginary and $|a|=1$. Then $a^2=-1$. This implies that the real span of $1$ and $a$ is a subfield of $\Q$ isomorphic
to the field of complex numbers. We denote this subfield by $\C(a)$. It is easy to prove that any subfield of $\Q$ containing real numbers and isomorphic to the field of complex numbers is of the form $\C(a)$
for some $a$ purely imaginary with $|a|=1$.

\medskip

The following was proved in \cite{CGr}.

\medskip

\begin{prop}\label{qua2}Let $a$ be as above. Then the centralizer of $a$ is $\C(a)$.
\end{prop}

\medskip

More generally, for any $\lambda \in \Q$, let $\C(\lambda)$ also denote the real span of $1$ and $\lambda$.

\begin{prop} \label{qua3} Let $\lambda \in \Q \setminus \R$. Then the centralizer of $\lambda$ is $\C(\lambda)$.
\end{prop}

\subsection{Hyperbolic spaces}

\noindent In this section, we discuss two models for the hyperbolic spaces, its isometry group and totally geodesics submanifolds.

\subsubsection{Projective Model}

\noindent We denote by $\F$ one of the real division algebras $\R$, $\C$, or $\Q$.
Let us write $\F ^{n+1}$ for a right $\F$ - vector space of dimension $n+1$.
The $\F$- projective space $\P\F^{n}$ is the manifold of right $\F$-lines in
$\F^{n+1}$. Let $\pi$ denote a natural projection
from $\F^{n+1} \setminus \{0\}$ to the projective space $\P\F^{n}$.

\medskip

Let $\F^{n,1}$ denote a $(n+1)$-dimensional $\F$-vector space
equipped with a Hermitian form $\Psi= \herm{-,-}$ of signature $(n,1)$.
Then there exists a (right) basis in $\F ^{n+1}$ such that the Hermitian
product is given by
$\herm{v,w}=v^*J_{n+1}w,$ where $v^*$ is the Hermitian transpose of
$v$ and  $J_{n+1}=(a_{ij})$ is the $(n+1)\times (n+1)$-matrix with
$a_{ij}=0$ for all $i \neq j$, $a_{ii}=1$ for all $i=1, \ldots, n$,
and $a_{ii}=-1$ when $i=n+1$.

\medskip

That is,
$$\herm{v,w}=\bar{v}_1 w_1+\ldots +  \bar{v}_nw_n -\bar{v}_{n+1}w_{n+1},$$
where $v_i$ and $w_i$ are coordinates of $v$ and $w$ in this basis. We call such a basis in $\F^{n,1}$ an {\sl orthogonal basis} defined by
a Hermitian form $\Psi= \herm{-,-}$.

\medskip

Let $V_{-}, V_0, V_{+}$ be the subsets of $\F^{n,1} \setminus \{0\}$
consisting of vectors where $\herm{v,v}$ is negative, zero, or
positive respectively. Vectors in $V_0$ are called {\sl null} or
{\sl isotropic}, vectors in $V_{-}$ are called {\sl negative}, an
vectors in  $V_{+}$ are called {\sl positive}. Their projections to $\P\F^{n}$
 are called {\sl isotropic}, {\sl negative}, and {\sl
positive} points respectively.

\medskip

The projective model of {\sl  hyperbolic space} ${\rm H^{n}_{\F}}$ is
the set of negative points in  $\P\F^{n}$, that is,
${\rm H^{n}_{\F}}=\pi(V_{-}).$

\vspace{2mm}

We will consider ${\rm H^{n}_{\F}}$ equipped with the Bergman metric \cite{CGr}:

$$d(p,q)=\cosh^{-1}\{|\Psi(v,w)| [\Psi(v,v)\Psi(w,w)]^{-1/2}\},$$

where $p,q \in{\rm H^{n}_{\F}}$, and $\pi(v)=p, \pi(w)=q$.

\medskip

The boundary $\partial
{\rm H^{n}_{\F}}=\pi(V_{0})$ of ${\rm H^{n}_{\F}}$  is the sphere formed by
all isotropic  points.

\medskip

Let ${\rm U}(n,1; \F)$ be the unitary group
corresponding to this Hermitian form $\Phi$. If $g \in {\rm U}(n,1; \F)$, then
$g(V_{-})=V_{-}$ and $g(v\lambda)=(g(v))\lambda$, for all $\lambda \in
\F$. Therefore ${\rm U}(n,1; \F)$ acts in $\P\F^{n}$, leaving ${\rm H^{n}_{\F}}$
invariant.

\medskip

The group ${\rm U}(n,1; \F)$ does not act effectively in ${\rm H^{n}_{\F}}$. The kernel of this
action is the center ${\rm Z}(n,1; \F)$. Thus, the projective group ${\rm PU}(n,1; \F)={\rm U}(n,1; \F)/{\rm Z}(n,1; \F)$ acts
effectively. The center $\Z(n,1, \F)$ in ${\rm U}(n,1; \F)$ is $\{\pm {\rm E}\}$ if $\F=\R$ or $\Q$, and is the circle group $\{\lambda {\rm E} : \mod \lambda =1\}$
if $\F=\C$. Here ${\rm E}$ is the identity  transformation of  $\F^{n,1}$.

\medskip

It is well-known, see for instance \cite{CGr}, that ${\rm PU}(n,1; \F)$
acts transitively in ${\rm H^{n}_{\F}}$ and doubly transitively on
$\partial{\rm H^{n}_{\F}}$.

\medskip

We remark that
\begin{itemize}
\item if $\F=\R$ then  ${\rm H^{n}_{\F}}$ is a real hyperbolic space ${\rm H^{n}_{\R}}$,

\item if $\F=\C$ then  ${\rm H^{n}_{\F}}$ is a complex hyperbolic space ${\rm H^{n}_{\C}}$,

\item if $\F=\Q$ then  ${\rm H^{n}_{\F}}$ is a quaternionic  hyperbolic space ${\rm H^{n}_{\Q}}$.
\end{itemize}

\medskip

It is easy to show \cite{CGr} that ${\rm H^{1}_{\Q}}$ is isometric to ${\rm H^{4}_{\R}}$.

\subsubsection{The ball model}\label{ball}

\noindent In this section, we consider the space $\F^{n,1}$ equipped by an orthogonal basis
$$e=\{e_1, \ldots, e_n, e_{n+1} \}.$$

For any $v \in \F^{n,1}$, we write $v=(z_1, \ldots, z_n, z_{n+1})$, where $z_i$, $i=1,\ldots, n+1$ are coordinates
of $v$ in this basis.

\medskip

If $v=(z_1, \ldots, z_n, z_{n+1}) \in V_{-}$, the condition $ \herm{v,v} < 0$ implies that $z_{n+1} \neq 0.$ Therefore, we may define a set of coordinates $w=(w_1, \ldots, w_n)$ in ${\rm H^{n}_{\F}}$ by $w_i(\pi(z))=z_i z_{n+1}^{-1}.$
In this way  ${\rm H^{n}_{\F}}$ becomes identified with the ball

$$\B=\B(\F)= \{w=(w_1, \ldots, w_n) \in \F^n : \Sigma_{i=1}^n |w_i|^2 < 1\}.$$

\medskip

With this identification the map $\pi: V_{-} \rightarrow {\rm H^{n}_{\F}}$ has the coordinate representation $\pi(z)=w$, where
$w_i= z_i z_{n+1}^{-1}$.

\subsubsection{Totally geodesic submanifolds}\label{tgsm}

We will need the following result, see \cite{CGr}, which describes all totally
geodesic submanifolds of ${\rm H^{n}_{\F}}$.

\medskip

Let $F$ be a subfield of $\F$.  An $F$-{\sl unitary subspace} of $\F^{n,1}$ is an $F$-subspace of $\F^{n+1}$
in which the Hermitian form $\Phi$ is $F$-valued. An $F$-{\sl hyperbolic subspace} of $\F^{n,1}$ is an $F$-unitary subspace in which
the Hermitian form $\Phi$ is non-degenerate and indefinite.

\medskip

\begin{prop}\label{ptgsm}
Let $M$ be a totally geodesic submanifold of ${\rm H^{n}_{\F}}$. Then either

\vspace{2mm}

(a) $M$ is the intersection of the projectivization of an $F$-hyperbolic subspace of $\F^{n,1}$ with ${\rm H^{n}_{\F}}$ for some
subfield $F$ of $\F$, or

\vspace{2mm}

(b) $\F=\Q$, and $M$ is a 3-dimensional  totally geodesic submanifold of a totally geodesic quaternionic line  ${\rm H^{1}_{\Q}}$  in ${\rm H^{n}_{\Q}}$.
\end{prop}

\medskip

From the last proposition follows that

\begin{itemize}
\item in the real hyperbolic space ${\rm H^{n}_{\R}}$ any totally geodesic submanifold  is isometric to ${\rm H^{k}_{\R}}, \ k=1, \ldots , n,$

\item in the  complex hyperbolic space ${\rm H^{n}_{\C}} $ any totally geodesic submanifod  is isometric to ${\rm H^{k}_{\C}},$ or to
${\rm H^{k}_{\R}},$ $\ k=1, \ldots , n,$

\item in the  quaternionic  hyperbolic space ${\rm H^{n}_{\Q}}$ any totally geodesic submanifold  is isometric to ${\rm H^{k}_{\Q}},$
 or to ${\rm H^{k}_{\C}},$ or to ${\rm H^{k}_{\R}},$
$k=1, \ldots , n,$ or to a 3-dimensional  totally geodesic submanifold of a totally geodesic quaternionic line  ${\rm H^{1}_{\Q}}$.
\end{itemize}

\medskip

In what follows we will use the following terminology:

\begin{itemize}
\item A totally geodesic submanifold of ${\rm H^{n}_{\Q}}$ isometric to ${\rm H^{1}_{\Q}}$ is called a {\sl quaternionic geodesic}.
\item A totally geodesic submanifold of ${\rm H^{n}_{\Q}}$ isometric to ${\rm H^{1}_{\C}}$ is called a {\sl  complex geodesic}.
\item A totally geodesic submanifold of ${\rm H^{n}_{\Q}}$ isometric to  ${\rm H^{2}_{\R}}$ is called a {\sl real plane}.
\end{itemize}

\medskip

It is clear that two distinct points in ${\rm H^{n}_{\Q}} \cup \partial
{\rm H^{n}_{\F}}$  span a unique quaternionic geodesic. We also remark that any $2$-dimensional totally geodesic submanifold of a totally geodesic
quaternionic line  ${\rm H^{1}_{\Q}}$ is isometric to ${\rm H^{1}_{\C}}$.

\begin{prop}
Let $V$ be a subspace of $\F^{n,1}$. Then each linear isometry of $V$ into $\F^{n,1}$ can be extended to an element of ${\rm U}(n,1; \F)$.
\end{prop}

\medskip

This is a particular case of the Witt theorem, see \cite{Sch}.

\medskip

\begin{co}
Let $S \subset {\rm H^{n}_{\F}}$ be a totally geodesic submanifold. Then each linear isometry of $S$ into ${\rm H^{n}_{\F}}$ can be extended to an
element of the isometry group of ${\rm H^{n}_{\F}}$.
\end{co}

\medskip

An interesting class of totally geodesic submanifolds of the quaternionic hyperbolic space ${\rm H^{n}_{\Q}}$ are submanifods which we call
totally geodesic submanifolds of complex type, or simply, submanifolds of complex type. Their construction is the following.
Let $\C^{n+1}(a) \subset\ \Q^{n+1} $ be the subset of vectors in $\Q^{n+1} $ with coordinates in $\C(a)$, where $a$ is a purely imaginary quaternion,
$\mod a =1$. Then $\C^{n+1}(a)$ is a vector space over the field $\C(a)$. The projectivization of $\C^{n+1}(a)$, denoted by ${\rm M}^n(\C(a))$, is a
projective submanifold of $\P\Q^{n}$ of real dimension $2n$. We call this submanifold  ${\rm M}^n(\C(a))$ a {\sl  projective submanifold of complex type of maximal dimension}.
It is clear that  the space $\C^{n+1}(a)$ is indefinite. The intersection
${\rm M}^n(\C(a))$  with ${\rm H^{n}_{\Q}}$ is a totally geodesic submanifold of  ${\rm H^{n}_{\Q}}$, called  a {\sl totally geodesic submanifold of complex type of maximal dimension}. It was proven in \cite{CGr} that all these submanifolds are isometric, and, moreover, they are globally equivalent
with respect to the isometry group of  ${\rm H^{n}_{\Q}}$, that is, for any two such submanifolds ${\rm M}$ and ${\rm N}$ there exists an element
$g \in  {\rm PU}(n,1; \Q)$ such that ${\rm M=g(N)}$. In particular, all of them are globally equivalent with respect to ${\rm PU}(n,1; \Q)$ to the
{\sl canonical totally geodesic complex  submanifold} ${\rm H^{n}_{\C}}$  defined by $\C^{n+1} \subset \Q^{n+1}$. This corresponds to the canonical
subfield of complex numbers $\C=\C(i) \subset  \Q$ in the above.

\medskip

If $V^{k+1} \subseteq  \C^{n+1}(a)$ is a subspace of complex dimension $k+1$, then its projectivization $W$ is called a {\sl projective submanifold of complex type of complex dimension k}. When $V^{k+1} \subseteq \C^{n+1}$, then its projectivization $W$ is called a {\sl canonical projective submanifold of complex type of complex dimension $k$} In this case, we will denote $W$ as $\P\C^{k}$.

\medskip

If  $V^{k+1} \subseteq  \C^{n+1}(a)$  is indefinite, then the intersection
of its projectivization with ${\rm H^{n}_{\Q}}$ is a totally geodesic submanifold of ${\rm H^{n}_{\Q}}$. We call this  submanifold of ${\rm H^{n}_{\Q}}$ a {\sl totally geodesic submanifold of complex type of complex dimension $k$}. When $V^{k+1} \subseteq  \C^{n+1}$, we call this totally geodesic submanifold a {\sl canonical totally geodesic  submanifold of complex type of complex dimension $k$}, or a {\sl canonical complex hyperbolic submanifold of dimension $k$} of ${\rm H^{n}_{\Q}}$.  In this case, we will denote this submanifold as ${\rm H^{k}_{\C}}$.

\subsubsection{A little more about the isometry group of the quaternionic hyperbolic space}\label{lmaig}

\noindent Let us consider the complex hyperbolic space ${\rm H^{n}_{\C}} $. It has a natural complex structure related to its isometry group, and the isometry group of ${\rm H^{n}_{\Q}}$
is generated by the holomorphic isometry group, which is the projective group ${\rm PU}(n,1; \C)$, and the anti-holomorphic isometry $\sigma$ induced by complex
conjugation in $\C ^{n+1}$. This anti-holomorphic isometry corresponds to the unique non-trivial automorphism of the field of complex numbers. Below we consider a similar isometry
of quaternionic hyperbolic space ${\rm H^{n}_{\Q}}$.

\medskip
We recall that if $f: \Q \rightarrow \Q$ is an automorphism of $\Q$, then $f$ is an inner automorphism of $\Q$, that is,
$f(q)= a q a^{-1}$ for some $a \in \Q, a \neq 0$.

\medskip

It follows from the fundamental theorem of projective geometry, see \cite{artin}, that each projective map $L: \P\Q^{n}\rightarrow \P\Q^{n}$ is induced by a semilinear or linear map $\tilde{L}: \Q^{n+1}\rightarrow \Q^{n+1}$.

\medskip

It is easy to see that if a projective map

$$L: \P\Q^{n}\rightarrow \P\Q^{n}$$

is induced by a semilinear map

$$\tilde{L}: \Q^{n+1}\rightarrow \Q^{n+1}, \tilde{L}(v)=a v a^{-1},v \in \Q^{n+1}, a\in \Q \setminus \R,$$

then it is  also induced by a linear map $v \mapsto av$.

\medskip

Therefore, the projective group of
$\P\Q^{n}$  is the projectivization of the linear group of  $\Q^{n+1}$.

\medskip

This implies that if

$$L: {\rm H^{n}_{\Q}} \rightarrow {\rm H^{n}_{\Q}} $$

is an isometry, then $L$ is induced  by a linear isometry

$$\tilde{L}: \Q^{n,1} \rightarrow \Q^{n,1}.$$

\medskip

This explains why the group of all isometries of ${\rm H^{n}_{\Q}}$
is the projectivization of the linear group  ${\rm U}(n,1; \Q)$, that is, ${\rm PU}(n,1; \Q).$

\medskip

Next we consider a curious map, which is an isometry of the quaternionic hyperbolic space, that has no analogue in geometries over
commutative fields. Let $\tilde{L}_a : v \mapsto av$, $v \in \Q^{n+1}$,  $a\in\Q$, $a$ is not real. The projectivization of this linear map
defines a non-trivial map $L_a:\P\Q^{n}\rightarrow \P\Q^{n}$. We remark that in projective spaces over commutative fields this map $L_a$ is identity.
It easy to see that $\tilde{L}_a \in {\rm U}(n,1; \Q)$ if and only if $|a|=1$, so $L_a$ in ${\rm PU}(n,1; \Q)$ if and only if $|a|=1$.

\begin{prop}\label{fixp} Let $a \in \Q$,  $|a|=1$. Then the fixed point set $S_a$ of $L_a$  is a totally geodesic submanifold of complex type of maximal dimension in ${\rm H^{n}_{\Q}}.$
This submanifold is globally equivalent to the canonical complex hyperbolic submanifold ${\rm H^{n}_{\C}}$ of ${\rm H^{n}_{\Q}}$.
\end{prop}

\begin{proof} The proof follows from Proposition \ref{qua3}.
\end{proof}

\medskip

It is also easy to see that if $a$ is purely imaginary, then  $L_a$ is an involution. We call this isometry $L_a$ a {\sl geodesic reflection} in $S_a$.

\section{Moduli of triples of points in quaternionic projective space}

\noindent In this section, we describe numerical invariants associated to an ordered triple of points in  $\P\Q^{n}$ which define the
equivalence class of a triple relative to the diagonal action of ${\rm PU}(n,1; \Q)$.

%%%%%%
\subsection{The Gram matrix}

\noindent Let $p=(p_1, \ldots ,p_m)$ be an ordered $m$-tuple of distinct
points in $\P\Q^{n}$ of quaternionic  dimension $n \geq 1$. Then we
consider a Hermitian quaternionic $m\times m$-matrix
$$G= \ G(p,v) \ = \ (g_{ij}) \ = \ (\herm{v_i,v_j}),$$
where $v=(v_1, \ldots, v_m)$, $v_i \in \Q^{n,1}$, $\pi(v_i)=p_i$, is a lift of $p$.

\medskip

We call $G$ a {\sl Gram matrix} associated to a $m$-tuple $p$
defined by $v$. Of course, $G$ depends on the chosen lifts $v_i$. When replacing $v_i$ by $ v_i \lambda_i,$ $\lambda_i \in \Q$, $\lambda_i \neq 0$, we get
$\tilde{G}=D^{*}\ G D,$ where $D$ is a diagonal quaternionic matrix,
$D=diag(\lambda_1,\ldots,  \lambda_m)$,

\medskip

We say that two  Hermitian quaternionic $m \times m$ - matrices $H$ and
$\tilde{H}$ are {\sl equivalent} if there exists a
diagonal quaternionic matrix $D=diag(\lambda_1,\ldots,  \lambda_m)$,
$\lambda_i\neq 0$, such that $\tilde{H}=D^{*} \ H \ D.$

\medskip

Thus, to each ordered $m$-tuple $p$ of distinct points in
$\P\Q^{n}$ is associated an equivalence class of Hermitian quaternionic $m\times
m$ - matrices.

\begin{prop}\label{gram1}
Let $p=(p_1, \cdots ,p_m)$  be an ordered $m$-tuple of distinct
negative points in  $\P\Q^{n}$. Then the equivalence class of Gram
matrices associated to $p$ contains a matrix $G=(g_{ij})$ such that
$g_{ii}=-1$ and $g_{1j}=r_{1j}$ are real positive numbers for
$j=2, \ldots , m$.
\end{prop}

\begin{proof}
Let $v=(v_1, \ldots , v_m)$ be a lift of $p$. Since the vectors
$v_i$ are negative, we have that $g_{ij} \neq 0$ for all $i,j=1, \ldots , m$, see, for instance, \cite{Sch}. First, by appropriate re-scaling, we may assume
that $g_{ii}=\herm{v_i,v_i}=-1$. Indeed, since $\herm{v_i,v_i}<0$, then
$\lambda_i =1/\sqrt{-\herm{v_i,v_i}} $ is well defined. Since $\lambda_i \in \R$, we have that
$$\herm{v_i \lambda_i , v_i \lambda_i}= \lambda_i^2 \herm{v_i,v_i}= \herm{v_i,v_i}/\mod{\herm{v_i,v_i}}=-1.$$
Then we get the result we need by
replacing the vectors $v_i$, $i=2, \ldots , m,$ if necessarily, by $ v_i \lambda_i$,
where
$$\lambda_i= \overline{\herm{v_1,v_i}}/\mod{\herm{v_1,v_i}}. $$
Indeed, since $\mod{\lambda_i}=1$, we have that $\herm{v_i \lambda_i , v_i \lambda_i}=-1$, $i=2, \ldots , m.$
On the other hand, for all $i>1$
$$\herm{v_1,v_i\lambda_i}=\herm{v_1,v_i}\lambda_i=\mod{\herm{v_1,v_i}}>0.$$

\end{proof}

\medskip

Let $p$ and $q$ be two points in $\P\Q^{n}$. We say that $p$ and $q$ are {\sl orthogonal} if $\herm{v,w}=0$ for some lifts $v$
and $w$ of $p$ and $q$ respectively. It is clear that if $p$ and $q$ are orthogonal then $\herm{v,w}=0$ for all lifts $v$ and $w$ of $p$ and $q$.

\medskip

Let $p=(p_1, \ldots ,p_m)$  be an ordered $m$-tuple of distinct points in $\P\Q^{n}$.  We call $p$ {\sl generic}
if $p_i$ and $p_j$ are not orthogonal for all $i,j = 1, \ldots ,m$, $i\neq j$.

\medskip

Let $G=(g_{ij})$ be a Gram matrix associated to $p$. Then $p$ is generic if and only if $g_{ij}\neq 0$ for all $i,j = 1, \ldots ,m$, $i\neq j$.

\medskip

\begin{prop}\label{gram2}
Let $p=(p_1, \cdots ,p_m)$  be an ordered generic $m$-tuple of distinct
positive points in  $\P\Q^{n}$. Then the equivalence class of Gram
matrices associated to $p$ contains a matrix $G=(g_{ij})$ such that
$g_{ii}=1$ and $g_{1j}=r_{1j}$ are real positive numbers for
$j=2, \ldots , m$.
\end{prop}

\begin{proof} The proof is a slight modification of the proof of Proposition \ref{gram1}.
\end{proof}

\medskip

It is easy to see that a matrix $G=(g_{ij})$ defined in Propositions \ref{gram1} and \ref{gram2} is unique. We call this matrix
$G$ a {\sl normal form} of the associated Gram matrix. Also, we call $G$ the {\sl normalized} Gram
matrix.

\medskip

We recall that a subspace $V \subset \F^{n,1}$ is called {\sl singular} or {\sl degenerate} if it contains at least one
non-zero vector that is orthogonal to all vectors in $V$. Otherwise, $V$ is called {\sl regular}.

\medskip

\begin{rmk} It is easy to see that if $V$ is singular then $V$ contains at least one isotropic vector and does not contain
negative vectors.
\end{rmk}

\begin{lemma}\label{gram3} Let $V=\{v_1, \ldots , v_m \}$ and $W=\{w_1, \ldots , w_m \}$ be two subspaces of $ \Q^{n,1}$ spanned by $v_i$ and $w_i$.
Suppose that $V$ and $W$ are regular, and $\herm{v_i,v_j} = \herm{w_i,w_j}$, for all $i, j = 1, \ldots  m$. Then the correspondence
$v_i \mapsto w_i$ can be extended to an isometry of $ \Q^{n,1}$.
\end{lemma}

\begin{proof} The proof follows from Theorem 1  in \cite{Hof}.
\end{proof}

\begin{prop}\label{grami}
Let $p=(p_1, \ldots ,p_m)$  and $p^\prime=(p_1 ^\prime, \ldots ,p_m ^\prime)$ be two ordered $m$-tuples of distinct
negative points in  $\P\Q^{n}$. Then $p$ and $p^\prime$  are congruent relative to the diagonal action of ${\rm PU}(n,1; \Q)$ if
and only if their associated Gram matrices are equivalent.
\end{prop}

\begin{proof}
Let $V$ and $V^\prime$ be the subspaces spanned by $v_i$ and $v_i^\prime $, $i=1, \ldots, m.$ Then it is clear that $V$ and $V^\prime$ are regular.
Since all the points $p_i$ are distinct, Lemma \ref{gram3} implies  that the map defined by $v \mapsto v^\prime$ extends to a linear isometry
 of $\Q^{n,1}$. The projectivization of this isometry maps $p$ in $p^\prime$.
\end{proof}

\begin{co}
Let $p=(p_1, \ldots ,p_m)$  and  $p^\prime=(p_1^\prime, \ldots ,p_m^\prime)$ be two ordered $m$-tuples of distinct
negative points in  $\P\Q^{n}$. Then $p$ and $p^\prime$ are congruent relative to the diagonal action of ${\rm PU}(n,1; \Q)$ if
and only if their normalized Gram matrices are equal.
\end{co}

\medskip

By applying the similar arguments, we get the following

\begin{prop}
Let $p=(p_1, \ldots ,p_m)$  and  $p^\prime=(p_1^\prime, \ldots ,p_m^\prime)$ be two ordered generic $m$-tuples of distinct
positive points in  $\P\Q^{n}$  such that the subspaces  $V$ and $V^\prime$ spanned by some lifts of $p$ and $p^\prime$  are regular.
Then $p$ and $p^\prime$ are congruent relative to the diagonal action of ${\rm PU}(n,1; \Q)$ if only if their associated Gram matrices are equivalent.
\end{prop}

\begin{co}\label{gramii} Let $p=(p_1, \ldots ,p_m)$  and  $p^\prime=(p_1^\prime, \ldots ,p_m^\prime)$ be two ordered generic $m$-tuples of distinct
positive points in  $\P\Q^{n}$ such that the subspaces  $V$ and $V^\prime$ spanned by some lifts of $p$ and $p^\prime$  are regular.
Then $p$ and $p^\prime$ are congruent relative to the diagonal action of ${\rm PU}(n,1; \Q)$ if and only if their normalized Gram matrices are equal.
\end{co}

\begin{rmk} It is easy to see that a subspace $V$ in $\Q^{n,1}$ is singular if and only if its projectivization
is a projective submanifold of $\P\Q^{n}$ tangent  to $\partial{\rm H^{n}_{\Q}}$ at an unique isotropic point $p$, lying, except this point $p$,
in the positive part of $\P\Q^{n}$.
\end{rmk}

%%%%%%
\subsection{Invariants of triangles in quaternionic hyperbolic geometry: inside and on the boudary}

\noindent In this section, we define some invariants of ordered triples of points in $\P\Q^{n}$
which generalize Cartan's angular invariant and Brehm's shape invariants in complex hyperbolic geometry
to quaternionic hyperbolic geometry. 

%%%%%%
\subsubsection {Quaternionic Cartan's angular invariant}\label{ideal}

\noindent First,  we recall the definition of Cartan' s angular invariant in complex hyperbolic geometry.
\medskip

Let  $p=(p_1, p_2, p_3)$ be an ordered triple of points in $\partial {\rm H^{n}_{\C}}$. Then Cartan's invariant $\A(p)$ of $p$ is defined as

$$\A(p)= \arg (- \herm{v_1,v_2}\herm{v_2,v_3}\herm{v_3,v_1}),$$
where $v_i$ is a lift of $p_i$.

It is easy to see that $\A(p)$ is well-defined,  that is, it is independent of the chosen  lifts, and it satisfies the inequality

$$-\pi/2 \leq \A(p)\leq\pi/2.$$

The inequalities above follow from the following proposition, see \cite{Gol}.

\begin{prop}\label{goldi} Let $ v, w, u \in \C^{2,1} $ be isotropic or negative vectors, then
$${\rm Re} (\herm{v,w}\herm{w,u}\herm{u,v}) \leq 0.$$
\end{prop}

\medskip

\begin{rmk} It is possible to extend the Cartan invariant of triples of isotropic points to triples of points in ${\rm H^{n}_{\C}} \cup \partial {\rm H^{n}_{\C}}$.
Indeed, no difficulty arises in the above definition because $\herm{v,w}\neq 0$ for any $v, w \in V_{0} \cup V_{-}.$
\end{rmk}

\medskip

Cartan's invariant is the only invariant of an ordered triple of isotropic points in the following sense:

\begin{prop}\label{car} Let $p=(p_1, p_2,p_3)$  and  $p^\prime=(p_1^\prime, p_2^\prime ,p_3^\prime)$ be two ordered triples of distinct points in $\partial {\rm H^{n}_{\C}}$.
Then $p$ and $p^\prime$ are congruent relative to the diagonal action of ${\rm PU}(n,1; \C$) if and only if $\A(p)=\A(p^\prime).$
\end{prop}

\medskip

The Cartan angular invariant $\A$ enjoys also the following properties, see \cite{Gol}:

\begin{enumerate}

\item If $\sigma$ is a permutation, then
$$\A(p_{\sigma (1)}, p_{\sigma (2)}, p_{\sigma (3)})={\rm sign}(\sigma)\A(p_1, p_2, p_3),$$

\item Let  $p=(p_1, p_2, p_3)$ be an ordered triple of points in $\partial {\rm H^{n}_{\C}}$. Then $p$ lies in the boundary
of a complex geodesic in ${\rm H^{n}_{\C}}$ if and only if $\A(p)=\pm \pi/2$,

\item Let  $p=(p_1, p_2, p_3)$ be an ordered triple of points in $\partial {\rm H^{n}_{\C}}$. Then $p$ lies in the boundary
of a real plane in ${\rm H^{n}_{\C}}$ if and only if $\A(p)=0,$

\item {\sl Cocycle property}. Let $p_1, p_2, p_3, p_4$ be points in  ${\rm H^{n}_{\C}} \cup \partial {\rm H^{n}_{\C}}$. Then
$$\A(p_1, p_2, p_3)+\A(p_1, p_3, p_4)=\A(p_1, p_2, p_4)+\A(p_2, p_3, p_4).$$

\item If $g \in {\rm PU}(n,1; \C)$ is a holomorphic isometry, then $\A(g(p)) =\A(p)$, and if $g$ is an anti-holomorphic isometry, then
$\A(g(p))=-\A(p).$

\end{enumerate}

\medskip

Next we define Cartan's angular invariant in quaternionic hyperbolic geometry.

\medskip

Let $v=(v_1, v_2, v_3)$ be an ordered triple of vectors in $\Q^{n,1}$. Then

$$H(v_1, v_2, v_3) =  \langle v_1, v_2, v_3 \rangle = \herm{v_1,v_2} \herm{v_2,v_3} \herm{v_3,v_1}$$
is called the {\sl Hermitian triple product}. An easy computation gives the following.

\begin{lemma}\label{car1}
Let $w_i= v_i \lambda_i$, $\lambda_i \in \Q$, $\lambda_i \neq 0$, then

$$H(w_1, w_2, w_3)= \langle w_1, w_2, w_3 \rangle = \overline{\lambda}_1 H(v_1, v_2, v_3) \lambda_1 \mod{\lambda_2}^2 \mod{\lambda_3}^2 =$$

$$\frac{\overline{\lambda}_1}{\mod{\lambda_1}}     H(v_1, v_2, v_3)   \frac{\lambda_1}{\mod{\lambda_1}}  \mod{\lambda_1}^2  \mod{\lambda_2}^2 \mod{\lambda_3}^2.$$

\end{lemma}

\medskip

\begin{co}\label{triples1} Let $p=(p_1, p_2, p_3)$ be an ordered triple of distinct points, $p_i \in \P\Q^{n}$.  Then there exists a lift $v=(v_1, v_2, v_3)$ of  $p=(p_1, p_2, p_3)$
such that $H(v_1, v_2, v_3)$ is a complex number.
\end{co}

\begin{proof} The proof follows by applying Lemma \ref{car1}, Proposition \ref{qsim}.
\end{proof}

\medskip

The historically first definition of Cartan's angular invariant in quaternionic hyperbolic geometry was given in \cite{AK}. In this paper, the authors defined
the quaternionic Cartan angular invariant

$$\A(p)=\A(p_1, p_2, p_3) $$
of an ordered triple  $p=(p_1, p_2, p_3)$ of distinct points, $p_i \in  {\rm H^{n}_{\Q}} \cup  \partial {\rm H^{n}_{\Q}},$ to be the angle between the quaternion $H(v_1, v_2, v_3)$ and the real line $\R \subset \Q$,  where $v_i$ is a lift of $p_i$.

\medskip

They proved that $\A(p)$ does not depend on the chosen lifts, and it is the only invariant of a triple of isotropic points in the above sense.

\medskip

Next, we represent a convenient formula to compute $\A(p).$

\medskip

Let $p=(p_1, p_2, p_3)$ be an ordered triple of distinct points,  $p_i \in  {\rm H^{n}_{\Q}} \cup  \partial {\rm H^{n}_{\Q}},$ and $v=(v_1, v_2, v_3)$
be a lift of $p=(p_1, p_2, p_3)$.  Then it follows from Proposition \ref{qsim} and Lemma \ref{car1} that

$$\A^*(p)= \arccos (-\frac{{\rm Re}(\ H(v_1, v_2, v_3))}{|H(v_1, v_2, v_3)|})$$
does not depend on the chosen lifts $v_i$.

\begin{prop} $\A(p)=\A^*(p)$.
\end{prop}

\begin{proof}  The proof follows from an easy computation.
\end{proof}

\medskip

The quaternionic Cartan angular invariant $\A(p)=\A(p_1, p_2, p_3) $ introduced above enjoys the following properties, see \cite{AK}:

\medskip

\begin{enumerate}
\item $0 \leq A(p)\leq\pi/2,$
\item If $\sigma$ is a permutation, then

$$\A(p_{\sigma (1)}, p_{\sigma (2)}, p_{\sigma (3)})=\A(p_1, p_2, p_3),$$

\item Let  $p=(p_1, p_2, p_3)$ be an ordered triple of points in $\partial {\rm H^{n}_{\Q}}$. Then $p$ lies in the boundary
of a quaternionic  geodesic in ${\rm H^{n}_{\Q}}$ if and only if $\A(p)=\pi/2$,
\item Let  $p=(p_1, p_2, p_3)$ be an ordered triple of points in $\partial {\rm H^{n}_{\Q}}$. Then $p$ lies in the boundary
of a real plane in ${\rm H^{n}_{\Q}}$ if and only if $\A(p)=0.$
\end{enumerate}

\medskip

It is seen that this quaternionic Cartan angular invariant, in contrast to the Cartan angular invariant in complex hyperbolic geometry, is non-negative, symmetric, and one can show that it does not enjoy the cocycle property.
We think that the definition of the Cartan angular invariant in quaternionic hyperbolic geometry given above does not explain well its relation
with the classical Cartan angular invariant in complex hyperbolic geometry. In what follows, we discuss another possible definitions of quaternionic Cartan's angular invariant and explain why the quaternionic Cartan angular invariant must be non-negative and symmetric.

\medskip

We start with the following simple fact which has a far reaching consequence for the construction of invariants of triples in quaternionic hyperbolic geometry. We think that this may also help for defining of invariants of triples in other hyperbolic geometries, for instance, in the hyperbolic octonionic plane.

\medskip

\begin{prop}\label{compt} Let $p=(p_1, p_2, p_3)$ be an ordered triple of distinct points, $p_i \in \P\Q^{n}$. Then there exists a projective submanifold $W \subset \P\Q^{n}$ of complex type of complex dimension $2$ passing through the points $p_i$, that is, $p_i \in W$, $i=1,2,3.$ Moreover, this submanifold $W$ can be chosen, up to the action of ${\rm PU}(n,1; \Q)$, to be the canonical complex submanifold $ \P\C^{2} \subset \P\Q^{n}$.
\end{prop}

\begin{proof} Let $p=(p_1, p_2, p_3)$ be an ordered triple of distinct points, $p_i \in \P\Q^{n}$,  and $v=(v_1, v_2, v_3)$
be a lift of $p=(p_1, p_2, p_3)$.

First, let us suppose that $p_i$ and $p_j$ are not orthogonal, for all $i \neq j$. Consider

$$H(v_1, v_2, v_3) =  \langle v_1, v_2, v_3 \rangle = \herm{v_1,v_2} \herm{v_2,v_3} \herm{v_3,v_1}.$$

It follows from Corollary \ref{triples1} and Lemma \ref{car1} that there exists $\lambda_1$ such that  $H(v_1\lambda_1, v_2, v_3)$ is a complex number.
Let us fix this $\lambda_1$, and let $w_1=v_1 \lambda_1$.  Then it follows from Lemma \ref{car1} that $ H(w_1, v_2, v_3)$ is a complex number for any lifts of
$p_2$ and $p_3$. Let now $\lambda_2 =  \herm{v_2,w_1}$,  $w_2=v_2 \lambda_2$. We have that
$\herm{w_1,w_2}$ is real.  Setting $\lambda_3 = \herm{v_3,w_1}$ and $w_3= v_3 \lambda_3$, we have that $\herm{w_3,w_1}$ is real.
Since $H(w_1, w_2, w_3) \in \C$, it follows that $ \herm{w_2,w_3} \in \C.$ Therefore for this normalizarion all Hermitian products are complex. This implies
that the complex span of $w_1, w_2, w_3$ is $\C$- unitary subspace of $\Q^{n,1}$ of dimension $3$, see Section \ref{tgsm}, and, therefore,
the points $p_1, p_2, p_3$ lie in a projective submanifold $W \subset  \P\Q^{n}$ of complex type of complex dimension $2$.

Now let us suppose that the set $\{p_1, p_2, p_3 \}$ contains orthogonal points. Assume, for example, that $p_1$ and $p_2$ are orthogonal, that is,
$\herm{v_1,v_2}=0$, where $v_1$ and $v_2$ are lifts of $p_1$ and $p_2$. Let $v_3$ be a lift of $p_3$. Setting  $\lambda_1=  \herm{v_1,v_3} $,  $\lambda_2 =  \herm{v_2,v_3}$,
and $w_1=v_1 \lambda_1$, $w_2=v_2 \lambda_2$, $w_3=v_3$, we have that all Hermitian products are real. It follows that the complex span of $w_1, w_2, w_3$
is $\C$- unitary subspace of $\Q^{n,1}$ of dimension $3$,  and, therefore,
the points $p_1, p_2, p_3$ lie in a projective submanifold $W \subset  \P\Q^{n}$ of complex type of complex dimension $2$.
The rest follows from the results of Section \ref{tgsm}.
\end{proof}

\medskip

\begin{co}\label{3in} Let $p=(p_1, p_2, p_3)$ be a triple of distinct negative points, $p_i \in {\rm H^{n}_{\Q}}$. Then $p$ lies in a totally
geodesic submanifold of ${\rm H^{n}_{\Q}}$ of complex type of complex dimension $2$.
\end{co}

\begin{co}\label{t1} Let $p=(p_1, p_2, p_3)$ be a triple of distinct isotropic points, $p_i \in \partial {\rm H^{n}_{\Q}}$. Then $p$ lies in the boundary
of a totally geodesic submanifold of ${\rm H^{n}_{\Q}}$ of complex type of complex dimension $2$.
\end{co}

\medskip

These results show that geometry of triples of points in  $\P\Q^{n}$ is, in fact, geometry of triples of points in  $\P\C^{2}$.
Therefore, all the invariants of triples of points in $ \P\Q^{n}$  relative to the diagonal action of  ${\rm PU}(n,1; \Q)$ can be constructed
using $2$-dimensional complex hyperbolic geometry.

\medskip

\begin{rmk} We notice that the results above  were used in the proof of the main result in \cite{AG}.
\end{rmk}

\medskip

First, we give a new definition of the Cartan angular invariant in quaternionic hyperbolic geometry.

\medskip

Let $p=(p_1, p_2, p_3)$ be an ordered triple of distinct isotropic points, $p_i \in \partial {\rm H^{n}_{\Q}}$. By Corollary \ref{t1},
we have that $p$ lies in the boundary of a totally geodesic submanifold $M$ of ${\rm H^{n}_{\Q}}$ of complex type of complex dimension $2$.
We know that $M$ is the projectivization of negative vectors in a $3$-dimensional complex subspace $V^3$ of  $\C^{n,1}$,
$\C^{n,1} \subset \Q^{n,1}$, and the boundary of ${\rm H^{2}_{\C}}$ is the projectivization of isotropic vectors in $V^3$.

Let $v=(v_1, v_2, v_3)$ be a lift of $p=(p_1, p_2, p_3)$. Then $v_i \in V^3$. We define

$$\A^{**}= \A^{**}(p)= \arg (- \herm{v_1,v_2}\herm{v_2,v_3}\herm{v_3,v_1}).$$

\medskip

Now we briefly explain how to use this invariant to classify ordered triples of isotropic points relative to the diagonal action of  ${\rm PU}(n,1; \Q)$.

\medskip

Let $p=(p_1, p_2,p_3)$  and  $p^\prime=(p_1^\prime, p_2^\prime ,p_3^\prime)$ be two ordered triples of distinct points in $\partial {\rm H^{n}_{\Q}}$.
Suppose that $\A^{**}(p)= \A^{**}(p^\prime)$. We will show that $p$ and $p^\prime$ are equivalent relative to the action of ${\rm PU}(n,1; \Q)$.

\medskip

By Corollary \ref{t1}, we have that $p$ is contained in the boundary of a totally geodesic submanifold $M(p)$ of ${\rm H^{n}_{\Q}}$ of complex type of complex dimension $2$. Also the same is true for $p^\prime$, where $p^\prime \in \partial M(p^\prime)$. We know, see Section \ref{tgsm}, that all such submanifolds are equivalent relative to the action of ${\rm PU}(n,1; \Q)$, therefore, there exists an element  $f \in {\rm PU}(n,1; \Q)$ such that $f(M(p)) = M(p^\prime)$. So, we can assume without loss of generality that $p$ and $p^\prime$ are in the boundary of the same submanifold $M = {\rm H^{2}_{\C}} \subset {\rm H^{n}_{\Q}}$. Then, by applying the classical result of Cartan, we have that there exists a complex hyperbolic isometry $g$ of $M$, $g \in {\rm PU}(2,1; \C)$, such that $g(p)=p^\prime$. This isometry $g$ can be extended to an element of ${\rm PU}(n,1; \Q)$ by the Witt theorem. This proves that $p$ and $p^\prime$ are equivalent relative the action of  ${\rm PU}(n,1; \Q)$.

\medskip

Next we show why it is more convenient to consider quaternionic  Cartan's angular invariant to be symmetric and non-negative (non-positive).

\medskip

We need the following lemma, see Lemma 7.1.7 in \cite{Gol}.
\medskip
\begin{lemma}\label{irv}
Let $p=(p_1, p_2,p_3)$ be an ordered triples of distinct points in $\partial {\rm H^{2}_{\C}}$. Then there exists a real plane $P \subset {\rm H^{2}_{\C}}$
such that inversion (reflection) $i_p$ in $P$ satisfies

$$i_P(p_1)=p_2,\ \ \    i_P(p_2)=p_1,\ \ \  i_P(p_3)=p_3.$$
\end{lemma}

\medskip

\begin{rmk} We recall that $i_P$ is an anti-holomorphic isometry of ${\rm H^{2}_{\C}}$ and  $$\A(p_2, p_1 p_3)  = - \A(p_1, p_2, p_3).$$
Also, as it is easy to see that any anti-holomorphic isometry of
${\rm H^{2}_{\C}}$ is a composition of an element of  ${\rm PU}(2,1; \C)$ and an anti-holomorphic reflection.
\end{rmk}

\begin{prop}Let $p=(p_1,p_2,p_3)$ be an ordered triples of distinct points in $\partial {\rm H^{n}_{\Q}}$, $n>1$. Then there exists an element
$f \in {\rm PU}(n,1; \Q)$ such that

$$f(p_1)=p_2, \ \ \ f(p_2)=p_1,\ \ \  f(p_3)=p_3.$$
\end{prop}

\begin{proof}
Repeating the arguments above, we can assume that $p$ is in the boundary of the $M = {\rm H^{2}_{\C}} \subset {\rm H^{n}_{\Q}}$.
Let $P \subset M = {\rm H^{2}_{\C}} $   be a real plane and $i_P$ be the reflection in $P$ acting in $M $ as in Lemma \ref{irv}. We will show that this map can be extended to an isometry of ${\rm H^{n}_{\Q}}$. Notice that $i_P$ as a map of $M$ is not induced by a linear map, it is induced by a semilinear map, therefore, we cannot apply the Witt theorem in this case.

Let $K$ be a totally geodesic submanifold of ${\rm H^{n}_{\Q}}$ isometric to  ${\rm H^{2}_{\Q}}$ which contains $M$.
An easy argument shows that there exists a totally geodesic submanifold $N$ of complex type of complex dimension $2$ in $K$ intersecting $M$ orthogonally along $P$. Let $i_N$ be the geodesic  reflection in $N$. We have that $i_N$ is an element of the isometry group of $K$, isomorphic to ${\rm PU}(2, 1; \C)$,  whose fixed point set is $N$. We notice that $i_N$ is induced by a linear map, therefore, it follows from the Witt theorem that $i_N$ can be extended to an isometry $f$ in ${\rm PU}(n,1; \Q)$. Note that by construction $f$ leaves $M$ invariant and its restriction to $M$ coincides with $i_P$. Therefore, $f(p_1, p_2,p_3)=(p_2, p_1, p_3)$. It is easy to see that the fixed point set of $f$ in ${\rm H^{n}_{\Q}}$ is a totally geodesic submanifold  of complex type of maximal dimension in ${\rm H^{n}_{\Q}}$. This submanifold is globally equivalent to the canonical submanifold ${\rm H^{n}_{\C}} \subset {\rm H^{n}_{\Q}}$.
\end{proof}

\medskip

\begin{co}
Let $p=(p_1, p_2,p_3)$  and  $p^\prime=(p_1^\prime, p_2^\prime ,p_3^\prime)$ be two ordered triples of distinct points in $\partial {\rm H^{n}_{\Q}}$.
Suppose that $\A^{**}(p)= -\A^{**}(p^\prime)$. Then $p$ and $p^\prime$ are equivalent relative to the diagonal action of ${\rm PU}(n,1; \Q)$.
\end{co}

\begin{rmk} This implies that if for two ordered triples of distinct isotropic points $p=(p_1, p_2,p_3)$ and $ p^\prime=(p_1^\prime, p_2^\prime ,p_3^\prime)$ we have that $|\A^{**}(p)|=|\A^{**}(p^\prime)|$,
then $p$ is equivalent to $ p^\prime$ relative to the diagonal action of group ${\rm PU}(n,1; \Q)$. Therefore, it is natural  to consider
instead of $\A^{**}$ its absolute value. Then the invariant $|\A^{**}|$ lies in the interval $[0,\pi /2]$ and is symmetric.
\end{rmk}

\begin{co} $|\A^{**}|  = \A^*$.
\end{co}

%%%%%%
\subsubsection{Quaternionic  Brehm's invariants}\label{neg}

\noindent First,  we recall the definition of the Brehm shape invariants in complex hyperbolic geometry.

\medskip

Let  $p=(p_1, p_2, p_3)$ be an ordered triple of distinct points in ${\rm H^{n}_{\C}}$ and $v=(v_1, v_2, v_3)$ be a lift of $p=(p_1, p_2, p_3)$.

\medskip

Brehm \cite{Bre} defined the invariant which he called the {\sl shape invariant}, or {\sl $\sigma$ - invariant}:

$$\sigma(p)=  {\rm Re} \frac{\herm{v_1,v_2} \herm{v_2,v_3} \herm{v_3,v_1}}{\herm{v_1,v_1}\herm{v_2,v_2}\herm{v_3,v_3}}.$$

\medskip

It is easy to check that $\sigma(p)$ is well- defined, that is, it does not depend on the chosen lifts,  and it is invariant relative the diagonal action of
the full isometry group of ${\rm H^{n}_{\C}}$.

\medskip

We consider $\{ p_1, p_2, p_3\}$ as the vertices of a triangle in hyperbolic space  ${\rm H^{n}_{\C}}$. Brehm \cite{Bre} showed that the side lengths and the shape
invariant are independent and characterize the triangle up to isometry.

\medskip

Now let  $p=(p_1, p_2, p_3)$ be an ordered triple of distinct points in ${\rm H^{n}_{\Q}}$, and $v=(v_1, v_2, v_3)$ be a lift of $p=(p_1, p_2, p_3)$.

\medskip

It is easy to check that if $w_i=v_i \lambda_i$, then

$$\herm{w_1,w_2} \herm{w_2,w_3} \herm{w_3,w_1}(\herm{w_1,w_1}\herm{w_2,w_2}\herm{w_3,w_3})^{-1} =$$

$$ \frac{\overline{\lambda}_1}{|\lambda_1|} {\herm{v_1,v_2} \herm{v_2,v_3} \herm{v_3,v_1}} (\herm{v_1,v_1}\herm{v_2,v_2}\herm{v_3,v_3})^{-1}\frac{\lambda_1}{|\lambda_1|}.$$

\medskip

This formula and Proposition \ref{qsim} imply that

$$\sigma^{*}(p)= - {\rm Re}({\herm{v_1,v_2} \herm{v_2,v_3} \herm{v_3,v_1}}[\herm{v_1,v_1}\herm{v_2,v_2}\herm{v_3,v_3}]^{-1})$$
is independent of the chosen lifts. Also, it is clear, that $\sigma^{*}(p)$ is invariant relative to the diagonal action of ${\rm PU}(n,1; \Q)$.

\medskip

We call this number $\sigma^{*}(p)$ the  {\sl quaternionic $\sigma$-shape invariant}.

\medskip

\begin{prop} $\sigma^{*}(p)\geq 0.$
\end{prop}
\begin{proof} By applying Corollary \ref{3in},  we can assume that $p$ lies  in  $M = {\rm H^{2}_{\C}} \subset {\rm H^{n}_{\Q}}$. Then the result
follows from Proposition\ref{goldi}.
\end{proof}

\medskip

As the first application of the results above, we have the following.

\begin{theorem} A triangle in ${\rm H^{n}_{\Q}}$ is determined uniquely up to the action of ${\rm PU}(n,1; \Q)$ by its three side lengths and
its  quaternionic $\sigma$-shape invariant $\sigma^{*}$.
\end{theorem}

\begin{proof}  Let  $p=(p_1, p_2, p_3)$ be an ordered triple of distinct points in ${\rm H^{n}_{\Q}}$. By applying Corollary \ref{3in}, we may assume that
$p$ lies  in  $M = {\rm H^{2}_{\C}} \subset {\rm H^{n}_{\Q}}$. Then the result follows from Proposition \ref{grami} and results in \cite{Bre}.
\end{proof}

\medskip

In  \cite{BrEt}, Brehm and Et-Taoui introduced another invariant in complex hyperbolic geometry which they called the {\sl direct shape invariant}, or, {\sl $\tau$-invariant}.
Below, we recall the definition of this invariant.

\medskip

Let  $p=(p_1, p_2, p_3)$ be an ordered triple of distinct points in ${\rm H^{n}_{\C}}$ and $v=(v_1, v_2, v_3)$ be a lift of $p=(p_1, p_2, p_3)$.
Then the {\sl direct shape invariant} is defined to be

$$\tau = \tau (p)= \frac{H(v_1, v_2, v_3)}{|H(v_1, v_2, v_3)|}.$$

\medskip

It is easy to check that $\tau (p)$ is independent of the chosen lifts.  Also, it was proved  in \cite{BrEt} that two triangles ${\rm H^{n}_{\C}}$ are equivalent relative to the diagonal action of ${\rm PU}(n,1; \C)$ if and only if the three corresponding edge lengths and the direct shape invariant
$\tau$ of the two triangles coincide.

\medskip

\begin{rmk}
Note that the $\sigma$-shape invariant is symmetric, but for the $\tau$-shape invariant we have that  $\tau(p_2, p_1, p_3)=\overline{\tau(p_1, p_2, p_3)}$.
This implies that the $\sigma$-shape invariant (with the side lengths) describes triangles up to the full isometry group of complex hyperbolic space (which includes anti-holomorphic isometries), but $\tau$-shape invariant (with the side lengths) describes  triangles up to the group of holomorphic
isometries  ${\rm PU}(n,1; \C)$.
\end{rmk}

\medskip

Now we define an analogue of $\tau$-shape invariant in quaternionic hyperbolic geometry. We start with the following lemma whose proof is based on a direct
computation.

\medskip

\begin{lemma}\label{prodt} Let  $p=(p_1, p_2, p_3)$ be an ordered triple of distinct points in ${\rm H^{n}_{\Q}}$ and $v=(v_1, v_2, v_3)$ be a lift of $p=(p_1, p_2, p_3)$.
Let $w_i= v_i \lambda_i$, $\lambda_i \in \Q$, $\lambda_i \neq 0$, then

$$H(w_1, w_2, w_3) |H(w_1, w_2, w_3)|^{-1}=$$

$$\frac{\overline{\lambda}_1}{ \mod{\lambda_1}} H(v_1, v_2, v_3) |H(v_1, v_2, v_3)|^{-1} \frac{\lambda_1}{\mod{\lambda_1}}.$$
\end{lemma}

\medskip

It is easy to see that $H(w_1, w_2, w_3) |H(w_1, w_2, w_3)|^{-1}$ is similar to $ H(v_1, v_2, v_3) |H(v_1, v_2, v_3)|^{-1}$ for any
$\lambda_i \in \Q$, $\lambda_i \neq 0$. Moreover, this similarity class contains a complex number, unique up to conjugation, see Proposition \ref{qsim}.

Let $ \tau^{*} (p)$ denote a unique complex number with non-negative imaginary part in this similarity class.

\medskip

We define the {\sl quaternionic $\tau$-shape invariant} to be $\tau^{*} = \tau^{*} (p)$. It is clear that $ \tau^{*} (p)$ does not depend on the chosen lifts.

\medskip

\begin{prop} Let $p=(p_1, p_2, p_3)$ and $p^\prime=(p_1^\prime, p_2^\prime ,p_3^\prime)$ be two ordered triples of distinct points in
 ${\rm H^{n}_{\Q}}$. Then these two triangles  are equivalent relative to the diagonal action of ${\rm PU}(n,1; \Q)$ if and only if the three corresponding edge lengths and the quaternionic  $\tau$-shape  invariant $\tau^{*}$ of the two triangles coincide.
\end{prop}

\begin{proof} By applying Corollary \ref{3in}, we may assume  that $p$ and $p^\prime$ are  in  $M = {\rm H^{2}_{\C}} \subset {\rm H^{n}_{\Q}}$
Then the result follows from Proposition \ref{grami}.
\end{proof}

\subsection{Moduli of triples  of positive points}\label{pos}

\noindent In this section, we describe the invariants associated to an ordered triple of positive points  in  $\P\Q^{n}$ which define the
equivalence class of the triple relative to the diagonal action of ${\rm PU}(n,1; \Q)$.

\medskip

First, we show how positive points in $\P\Q^{n}$ are related to totally geodesic submanifolds in  ${\rm H^{n}_{\Q}}$ isometric to  ${\rm H^{n-1}_{\Q}}$.
We call such submanifolds of ${\rm H^{n}_{\Q}}$ {\sl totally geodesic quaternionic  hyperplanes} in ${\rm H^{n}_{\Q}}$.

\medskip
We recall that $\pi$ denotes a natural projection
from $\Q^{n+1} \setminus \{0\}$ to the projective space $\P\Q^{n}$.

\medskip

If $H \subset \rm H^{n}_{\Q}$ is a totally geodesic quaternionic hyperplane, then $H =\pi(\tilde{H}) \cap \rm H^{n}_{\Q}$, where $\tilde{H} \subset \Q^{n,1}$ is a quaternionic linear
hyperplane. Let $\tilde{H}^\perp$ denote the orthogonal complement of $\tilde{H}$ in $\Q^{n,1}$ with respect to the Hermitian form $\Phi$. Then   $\tilde{H}^\perp$ is a positive quaternionic
line, and $\pi(\tilde{H}^\perp)$ is a positive point in $\P\Q^{n}$. Thus, the totally geodesic  quaternionic  hyperplanes in ${\rm H^{n}_{\Q}}$ bijectively correspond to positive points.
We call $p= \pi(\tilde{H}^\perp)$ the {\sl polar point} of  a totally geodesic quaternionic hyperplane  $H$. So, the invariants associated to an ordered triple of positive points  in  $\P\Q^{n}$
are invariants of an ordered triple of totally geodesic quaternionic hyperplane in ${\rm H^{n}_{\Q}}$.

\medskip

Let $H_1$ and $H_2$ be distinct totally geodesic quaternionic hyperplanes in $\rm H^{n}_{\Q}$ and $p_1$, $p_2$ be their polar points.
Let  $v_1$ and $v_2$ in $\Q^{n,1}$ be their lifts. Then we define

$$d(H_1,H_2)= d(p_1,p_2) = \frac{\herm{v_1,v_2}\herm{v_2,v_1}}{\herm{v_1,v_1}\herm{v_2,v_2}}.$$
It is easy to see that $d(H_1,H_2)$ is independent of the chosen
lifts of $p_1, p_2$, and that $d(H_1,H_2)$ is invariant with respect to
the diagonal action of ${\rm PU}(n,1; \Q)$.

\medskip
There is no accepted name for this invariant in the literature.
It is not difficult to show using standard arguments that the distance or the angle between $H_1$ and $H_2$ is given in
terms of $d(H_1,H_2)$, (see, for instance, the Goldman \cite{Gol} for the case of complex hyperbolic geometry). So, we will call this invariant $d(H_1,H_2)$ the
{\sl distance-angular invariant} or, simply, {\sl d-invariant}.

\medskip

Also, it is easy to see that:

\begin{itemize}\item $H_1$ and $H_2$ are concurrent if
and only if $d(H_1,H_2)<1$,
\item $H_1$ and $H_2$ are asymptotic if and
only if $d(H_1,H_2)=1$,
\item $H_1$ and $H_2$ are ultra-parallel if and
only if $d(H_1,H_2)>1$.
\end{itemize}

Moreover,  $d(H_1,H_2)$ is the only invariant of an ordered pair of totally geodesic quaternionic hyperplanes in $\rm H^{n}_{\Q}$. We have also that the angle $\theta$ between
$H_1$ and $H_2$ (in the case $d(H_1,H_2)<1$) is given by

$$\cos^{2}(\theta)= d(H_1,H_2),$$
and the distance $\rho$ between $H_1$ and $H_2$ (in the case $d(H_1,H_2)\geq 1$) is given by
$$\cosh^2(\rho)= d(H_1,H_2)$$.

We remark that $0 <\theta \leq \pi/2$.

\medskip

We say that $H_1$ and $H_2$ are {\sl
orthogonal} if $\theta = \pi/2$, this is equivalent to the equality
$d(H_1,H_2)=0.$

\medskip

Now let $(H_1, H_2, H_3)$ be an ordered triple of distinct  totally geodesic quaternionic hyperplanes in $\rm H^{n}_{\Q}$.
Let $p_1$, $p_2$, $p_3$ be the polar points of $H_1, H_2, H_3$ and $v_1$, $v_2$,  $v_3$ be their lifts in $\Q^{n,1}$.
Then, by applying Proposition \ref{compt}, we can assume that $p_1$, $p_2$, $p_3$ lie in a projective submanifold $W \subset \P\Q^{n}$ of complex type of complex dimension $2$ passing through the points $p_i$.
Moreover, this submanifold $W$ can be chosen, up to the action of ${\rm PU}(n,1; \Q)$, to be the canonical complex submanifold $ \P\C^{2} \subset \P\Q^{n}$. Therefore, we can assume without loss of
generality that the coordinates of the vectors  $v_1$, $v_2$,  $v_3$ are complex numbers.

\medskip

Let $G=\ (g_{ij}) = (\herm{v_i,v_j})$ be the Gram matrix associated to the points $p_1$, $p_2$, $p_3$ defined by the chosen vectors  $v_1$, $v_2$,  $v_3$ as above.
Then it follows from Proposition \ref{gram2} and the proof of Proposition \ref{compt} that $g_{ii}=1$,  $g_{1j}=r_{1j}\geq 0$, and $g_{23}=r_{23}e^{\alpha}$.
We call such a matrix $G$ a {\sl complex normal form} of the associated Gram matrix. Also, we call $G$ the {\sl complex normalized} Gram matrix.

\medskip

Next we construct the moduli space of ordered triples of distinct  totally geodesic quaternionic hyperplanes in $\rm H^{n}_{\Q}$.
We consider only the regular case, that is, when for all triples in question the spaces spanned by lifts of their polar points are regular, see Corollary \ref{gramii}.
It is easy to see that in non-regular case, the totally geodesic quaternionic hyperplanes $H_1, H_2, H_3$ are all asymptotic, that is, $d(H_i,H_j)=1$, $i\neq j$.
It was shown in \cite{CDGT} that a similar problem, the congruence  problem for triples of complex geodesic in complex hyperbolic plane, cannot be solved by using
Hermitian invariants.

\medskip

An ordered triple $H=(H_1, H_2, H_3)$ of totally geodesic quaternionic hyperplanes in $\rm H^{n}_{\Q}$  is said to be {\sl generic} if
$H_i$ and $H_j$ are not orthogonal for all $i,j=1,2,3.$ It is clear that $H$
is generic if and only if the corresponding triple of polar points is generic.

\medskip

We start with the following proposition:

\begin{prop}\label{pos1} Let $H=(H_1, H_2, H_3)$ and  $H^\prime = (H^\prime_1, H^\prime_2, H^\prime_3)$ be two ordered generic triples of distinct  totally geodesic quaternionic hyperplanes in $\rm H^{n}_{\Q}$.
Let $p=(p_1, p_2, p_3)$ and  $p^\prime = (p^\prime_1,p^\prime_2, p^\prime_3)$ be their polar points. Let  $v=(v_1,v_2, v_3)$ and $v^\prime = (v^\prime_1,v^\prime_2,v^\prime_3)$ be their lifts in $\Q^{n,1}$
such that the Gram matrices $G=(g_{ij})= (\herm{v_i,v_j})$ and $G^\prime=(g^\prime_{ij})= (\herm{v^\prime_i,v^\prime_j})$ are complex normalized. Suppose that the spaces spanned by
$v_1$, $v_2$,  $v_3$ and $v^\prime_1$, $v^\prime_2$,  $v^\prime_3$ are regular. Then $H$ and $H^\prime$ are equivalent relative to
the diagonal action of ${\rm PU}(n,1; \Q)$ if and only if $G=G^\prime$ or $\bar{G}=G^\prime$.
\end{prop}

\begin{proof} If $G=G^\prime$, then it  follows from Corollary \ref{gramii} that there exists a linear isometry $L : \Q^{n,1} \longrightarrow \Q^{n,1}$ such that $L(v_i)=v^\prime_i$.
Let us suppose that $\bar{G}=G^\prime$. Then we have that $g_{1j}=g^\prime_{1j}$ because $g_{1j}$ and $g^\prime_{1j}$ are real. Also, if $g_{23}= r_{23}e^{\alpha}$,
then $g^\prime_{23}= r_{23}e^{-\alpha}$.
Let us consider the semi-linear map $L_j : \Q^{n,1} \longrightarrow \Q^{n,1}$ defined by the rule $L_j (v)=jvj^{-1}.$
Then we have that
$$\herm{L_j(v_k),L_j(v_l)} = \herm{jv_kj^{-1},jv_lj^{-1}} =  \overline{j^{-1}}\herm{jv_k,jv_l}j^{-1} $$

$$= -\bar{j}\herm{v_k,v_l}j^{-1}=j\herm{v_k,v_l}j^{-1}=\overline{\herm{v_k,v_l}}.$$

Here we have used that $jzj^{-1}=\bar{z}$ for any complex number $z$ and that $j^{-1}=-j$.

\medskip

It follows that the Gram matrix of the vectors $L_j(v_k),$ $k=1,2,3,$ is equal to $\bar{G}$. Therefore, using the first part of the proof, we get that the triples $v$ and $v^\prime$
are equivalent relative to the diagonal action of ${\rm U}(n,1; \Q)$. This implies that $H$ and $H^\prime$ are equivalent relative to
the diagonal action of ${\rm PU}(n,1; \Q).$
\end{proof}

\medskip

Let $p=(p_1,p_2,p_3)$  be an ordered generic triple of distinct positive points in  $\P\Q^{n}$. Let $v=(v_1,v_2, v_3)$ be their lifts in $\Q^{n,1}.$
Let $H(v_1,v_2,v_3)=(\herm{v_1,v_2}\herm{v_2,v_3}\herm{v_3,v_1}) \in \Q.$

\medskip

Proposition \ref{pos1} justifies the following definition.

\medskip

We define the {\sl angular invariant} $A=A(p)$ of an ordered generic triple of distinct positive points $p=(p_1,p_2,p_3)$ in $\P\Q^{n}$
to be the argument of the unique complex number $b=b_0 +b_1$ with $b_1 \geq 0$ in the similarity class of

$$\tau( v_1,v_2,v_3) = H(v_1,v_2,v_3)|H(v_1,v_2,v_3)|^{-1}.$$

\medskip

It follows from Lemma \ref{prodt} and Corollary \ref{qua1} that $A=A(p)$ is defined uniquely by the real part of $\tau( v_1,v_2,v_3)$ which does not depend on the chosen lifts $v_1,v_2,v_3.$

\medskip
It is clear from the construction that $A=A(p)$ is invariant with respect to the diagonal action of ${\rm PU}(n,1; \Q)$. Also, $0 \leq A(p) \leq \pi.$

\medskip

By applying the proof of Proposition \ref{compt}, we can chose lifts $v=(v_1,v_2, v_3)$ of  $p=(p_1,p_2,p_3)$ such that the Gram matrix associated to $p=(p_1,p_2,p_3)$ defined by $v=(v_1,v_2, v_3)$
is the  complex normalized Gram matrix of $p=(p_1,p_2,p_3)$, that is, $g_{ii}=1$,  $g_{1j}=r_{1j}> 0$, and $g_{23}=r_{23}e^{\alpha}$, $r_{23}>$.
It is clear that for these lifts $A(p)= \alpha$.

\medskip

Now we are ready to describe the moduli space of ordered triples of distinct regular generic totally geodesic quaternionic hyperplanes in $\rm H^{n}_{\Q}$.

\begin{theorem} \label{pos2}Let $H=(H_1, H_2, H_3)$ and  $H^\prime = (H^\prime_1, H^\prime_2, H^\prime_3)$ be two ordered  distinct regular generic  totally geodesic quaternionic hyperplanes in $\rm H^{n}_{\Q}$.
Then  $H=(H_1, H_2, H_3)$ and  $H^\prime = (H^\prime_1, H^\prime_2, H^\prime_3)$ are equivalent relative to the diagonal action of  ${\rm PU}(n,1; \Q)$ if and only if
$d(H_i,H_j)=  d(H^\prime_i,H^\prime_j)$ for all $i<j$, $i,j=1,2,3,$ and $A(p)=A(p^\prime)$, where $p=(p_1,p_2,p_3)$ and $p=(p^\prime_1, p^\prime_2,p^\prime_3)$ are the polar points of
$H)$ and  $H^\prime$.
\end{theorem}

\begin{proof} We can choose lifts $v=(v_1,v_2, v_3)$ and $v^\prime = (v^\prime_1,v^\prime_2,v^\prime_3)$ of
$p=(p_1,p_2,p_3)$ and $p=(p^\prime_1, p^\prime_2,p^\prime_3)$ such that the Gram matrices $G$ and $G^\prime$ associated to $p$ and $p\prime$ defined by $v$ and $v^\prime$
are complex normalized. Then it follows from the definition of the Gram matrix that $d_{ij}= d(p_i,p_j)=g_{ij}g_{ji}=|g_{ij}|^2,$
and that

$$A(p)= A(p_1,p_2,p_3)=\arg(g_{12} \ g_{23} \ g_{31})=\arg (r_{12} \ g_{23j} \ r_{31}).$$

The first equality implies that $|g_{ij}|=\sqrt {d_{ij}}$. Since $H$
is generic, we have that $r_{1j} > 0$ for all $j>1$, and that
$g_{23}\neq 0.$ Therefore, the second equality implies that
$A(p)= \arg ( g_{23})$. Thus, all the entries of the complex normalized
Gram matrix $G(p)$ of $p$ are recovered uniquely in terms of the
invariants $d_{ij}$ and $A(p)$ above. Now the proposition follows
from Corollary \ref{gramii}.
\end{proof}

\medskip

Next, as a corollary of Theorem \ref{pos2}, we give an explicit
description of the moduli space of regular generic triples of totally geodesic quaternionic hyperplanes in $\rm H^{n}_{\Q}$.
First of all, it follows from Theorem \ref{pos2} that ${\rm PU}(n,1; \Q)$-congruence class of an
ordered regular generic triple $H$ of distinct generic  totally geodesic quaternionic hyperplanes in $\rm H^{n}_{\Q}$ complex is described uniquely by three
d-invariants $d_{12},d_{13},d_{23}$ and the angular invariant
$\alpha= \A(p_1, p_2 , p_3)$. Now, let $G=(g_{ij})$ be the complex
normalized Gram matrix of $H$. Then $g_{ii}=1, \  g_{1j}=r_{1j}>0$,
and $\arg (g_{23})=A(p_1, p_2 , p_3)$. We have that $d_{1j}=r_{1j}^2$ and
$d_{23}=r_{23}^2 $. Also,

$$g_{23}=r_{23} \ e^{i\alpha}=r_{23}(\cos\alpha + i\sin\alpha).$$

A straightforward computation shows that

$$\det G = 1-(r_{12}^2 + r_{13}^2 +r_{23}^2)+ 2r_{12} r_{13}r_{23}\cos\alpha.$$

Using the lexicographic order, we define $r_1=\sqrt d_{12}, \ r_2=\sqrt d_{13}, \ r_3=\sqrt d_{23}.$

\medskip

It follows from the Silvester criterium that $\det G \leq 0$, therefore, we have

\medskip

\begin{co} The moduli space ${\mathcal M}_{0}(3)$  of regular generic  totally geodesic quaternionic hyperplanes in $\rm H^{n}_{\Q}$
is homeomorphic to the set
$${\mathbb{M}}_{0}(3)=\{(r_1,r_2,r_3,\alpha) \in \R^4 :r_i>0,
\alpha \in (o,\pi], \ 1-(r_{1}^2 + r_{1}^2 +r_{3}^2)+ 2r_{1}
r_{2} r_{3}\cos\alpha \leq0\}.$$

\end{co}

\begin{rmk}  If the triple $H=(H_1, H_2, H_3)$ is not generic, we need less invariants to describe the equivalence class of $H$. For instance, if
$H_2$ and $H_3$ are both orthogonal to $H_1$, we need only $d(H_2,H_3).$
\end{rmk}

 %$\P\Q^n$

\subsection{Moduli of triples  of points: mixed configurations}

\noindent As we know from  Section \ref{neg},  that, up to isometries of ${\rm H^{n}_{\Q}}$, triples of points of  ${\rm H^{n}_{\Q}}$,
that is, triangles, are characterized by the side lengths and the Brehm shape invariant. Also, triples of points in $\partial {\rm H^{n}_{\Q}}$, that is, ideal triangles,
are characterized by the Cartan angular invariant, see Section \ref{ideal}. The description of the moduli of triples of positive points was given in Section \ref{pos}.

\medskip

In this section, we describe the invariants for mixtures of the three types of points in $\P\Q^{n}$ relative to the diagonal action of  ${\rm PU}(n,1; \Q)$.
Using this, we describe the  moduli of the corresponding configuration.

\medskip

First, we give a list of all the  triples of points in  $\P\Q^{n}$.

\medskip

Let $p=(p_1, p_2, p_3)$ be a triple of points in $\P\Q^{n}$. Then we have the following possible configurations (up to permutation).

\begin{enumerate}

\item Ideal triangles: $p_i$ is isotropic for $i=1, 2, 3.$

\item Triangles: $p_i$ is negative, that is, $p_i$ in ${\rm H^{n}_{\Q}}$ for $i=1, 2, 3.$

\item Triangles of totally geodesic quaternionic hyperplanes in $\rm H^{n}_{\Q}$.

\item Triangles with two ideal vertices: $p_1$ and $p_2$ are isotropic, and (a) $p_3$ is negative, (b) $p_3$ is positive.

\item Triangles with one ideal vertex: $p_1$ is isotropic, and (a) $p_2$, $p_3$ are negative, (b) $p_2$, $p_3$ are positive,
(c) $p_2$ is negative, $p_3$ is positive.

\item Triangles with one negative vertex and two positive vertices: $p_1$ is negative, and $p_2$, $p_3$ are positive.

\item Triangles with one positive vertex and two negative vertices: $p_1$ is positive, and $p_2$, $p_3$ are negative.

\end{enumerate}

\medskip

The moduli of triangles in the first three cases have been already established in the previous sections  Below, we describe invariants of mixed configurations 4-7.

\medskip

First, we recall the definition of so called  $\eta$-invariant \cite{Gol} in complex hyperbolic geometry.

\medskip

Let $(p_1, p_2, q)$ be an ordered  triple of points in $\P\C^{n}$. We suppose that $p_1$ and $p_2$ are isotropic and $q$ is positive.
Let $(v_1, v_2, w)$ be a lift of $(p_1, p_2, p_3)$. Then it is easy to check that the complex number
$$\eta(v_1,v_2, w)= \frac{\herm{v_1,w}\herm{w,v_2} }{ \herm{v_1,v_2} \herm{w,w}}$$
is independent of the chosen lifts and will be denoted by $\eta(p_1,p_2, q)$.

\medskip

We call the number $\eta(p_1,p_2, q)$ the Goldman {\sl $\eta$-invariant}. This invariant was introduced by Goldman \cite{Gol} to study the intersections of bisectors
in complex hyperbolic space. Later, Hakim and Sandler \cite{HS1} generalized the  Goldman construction for more general triples of points.

\medskip

The aim of this section is to define analogous invariants in quaternionic hyperbolic geometry and to prove the congruence theorems for all triples in question.

\subsubsection{ Triangles with two ideal vertices}

Let $p=(p_1, p_2, p_3)$ be a triple of points in $\P\Q^{n}$. In what follows, we
will denote by $v$ the triple $v=(v_1,v_2, v_3$, where $v_i$ is a lift of $p_i$ in $\Q^{n,1}$.

First we consider the case when  $p_1$ and $p_2$ are isotropic and $p_3$ is positive. This configuration was considered by Goldman \cite{Gol} in complex hyperbolic geometry. Geometrically it can be considered as a configuration of two points in the boundary of quaternionic hyperbolic space ${\rm H^{n}_{\Q}}$ and a totally geodesic quaternionic hyperplane in $\rm H^{n}_{\Q}$.

\medskip

Let $v_1$, $v_2$ be isotropic vectors representing $p_1, p_2$ and $v_3$ a positive vector representing $p_3$.

Consider the following quaternion:

$$\eta(v_1,v_2, v_3)= \herm{v_1,v_3}\herm{v_3,v_2}\herm{v_1,v_2}^{-1}{ \herm{v_3,v_3}}^{-1}.$$

\medskip

Now let us take another lifts of $p=(p_1, p_2, p_3)$: $v^\prime_1 = v_1 \lambda_1$, $v^\prime_2 = v_2 \lambda_2$, $v^\prime_3 = v_3 \lambda_3$.

\medskip

It is easy to check that

$$\eta(v_1\lambda_1, v_2\lambda_2, v_3\lambda_3)= \frac{\bar{\lambda}_1}{|\lambda_1|}\eta(v_1,v_2, w)\frac{\lambda_1}{|\lambda_1|}.$$

\medskip

Therefore, this implies that $\eta(v_1,v_2, v_3)$ is independent of the choices of lifts of $p_2$ and $p_3$, and if we change a lift of $p_1$, we get a similar quaternion.

\medskip

By applying Proposition \ref{compt}, we can assume that $p_1$, $p_2$, $p_3$ lie in a projective submanifold $W \subset \P\Q^{n}$ of complex type of complex dimension $2$ passing through the points $p_i$.
Moreover, this submanifold $W$ can be chosen, up to the action of ${\rm PU}(n,1; \Q)$, to be the canonical complex submanifold $ \P\C^{2} \subset \P\Q^{n}$. Therefore, we can assume without loss of
generality that the coordinates of the vectors  $v_1$, $v_2$,  $v_3$ are complex numbers.

\medskip
Let $G=\ (g_{ij}) = (\herm{v_i,v_j})$ be the Gram matrix associated to the points $p_1$, $p_2$, $p_3$. Then $g_{11}=0$, $g_{22}=0$.
Since $g_{33}= \herm{v_3,v_3}>0$, replacing $v_3$ by $v_3 a$, where $a= 1/\sqrt{g_{33}}$, we may assume that $g_{33}=1$.
Replacing $v_2$ by $v_2 b$, where $b=1/\herm{v_1,v_2}$, we may assume that $g_{12}=1.$ We keep the same notations for the modified vectors.
After that, replacing $v_1$ by $v_1 b$, where $b=  1/\herm{v_3,v_1}$ and $v_2$ by $v_2 c$, where $c= \herm{v_1,v_3}$, we get $g_{12}=g_{13}=1$.
So, after this re-scaling we get that $G=(g_{ij}$ has the following entrances:
$g_{11}=g_{22}=0, g_{12}=g_{13}=1, g_{33}=1,$ $g_{23}$ is an arbitrary complex number.

\medskip
We call such a matrix $G$ a {\sl complex normal form} of Gram matrix associated to $p=(p_1, p_2, p_3)$. Also, we call $G$ the {\sl complex normalized} Gram matrix.

\medskip

It is clear that any vectors $v_1$, $v_2$ and $v_3$ which represent  $p_1$, $p_2$, and $p_3$ generate a regular space in $\Q^{n,1}.$
Therefore, repeating almost word for word the proof of Proposition \ref{pos1}, we get

\begin{prop} Let $p$ and  $p^\prime = (p^\prime_1, p^\prime_2, p^\prime_3)$ as above.
Let  $v$ and $v^\prime$ be their lifts such that the Gram matrices $G=(g_{ij})= (\herm{v_i,v_j})$ and $G^\prime=(g^\prime_{ij})= (\herm{v^\prime_i,v^\prime_j})$ associated to $p$ and $p^\prime$ are complex normalized.
Then $p$ and $p^\prime$ are equivalent relative to the diagonal action of ${\rm PU}(n,1; \Q)$ if and only if $G=G^\prime$ or $\overline{G}=G^\prime$.
\end{prop}

\medskip

Again this justifies the following definition.

\medskip

Let $p=(p_1, p_2, p_3)$ be a triple of points as above.
Then the {\sl quaternionic $\eta$-invariant},  $\eta =\eta (p)$, is defined to be the unique complex number $b=b_0 +b_1$ with $b_1 \geq 0$,
in the similarity class of $\eta(v_1,v_2, v_3)$.

\medskip

\begin{theorem}\label{mix1} Let $p$ and  $p^\prime$ as above. Then $p$ and $p^\prime$ are equivalent relative to the diagonal action of  ${\rm PU}(n,1; \Q)$
if and only if $\eta (p) =\eta (p^\prime)$.
\end{theorem}
\begin{proof} It follows from the above that we can choose lifts $v=(v_1,v_2, v_3)$ and $v^\prime = (v^\prime_1,v^\prime_2,v^\prime_3)$ of
$p=(p_1,p_2,p_3)$ and $p^\prime = (p^\prime_1, p^\prime_2,p^\prime_3)$ such that the Gram matrices $G$ and $G^\prime$ associated to $p$ and $p^\prime$ defined by $v$ and $v^\prime$
are complex normalized. Then it follows from the definition of the Gram matrix that $g_{23} = \overline{\eta (p)}$. This implies that $G =G^\prime$. Now the proof follows from Lemma \ref{gram1}.
\end{proof}

\medskip

The case when $p_3$ is negative is similar. The $\eta$-invariant is defined by the same formula and  the proof of the congruence theorem is a slight modification of Theorem \ref{mix1}.
Therefore, we have the following theorem.

\begin{theorem}\label{mix2} Let $p=(p_1, p_2, p_3)$, where $p_1, p_2$ are isotropic and $p_3$ is negative. Then the congruence class of $p$  relative to the diagonal action of  ${\rm PU}(n,1; \Q)$
is completely defined by the Goldman invariant $\eta (p)$.
\end{theorem}

\medskip

We remark that this problem was considered by Cao, \cite{Cao}, see Theorem 1.1 item (ii). In order to describe the congruence class  relative to the diagonal action of  ${\rm PU}(n,1; \Q)$ of a triple  $p=(p_1, p_2, p_3)$,
where $p_1, p_2$ are isotropic and $p_3$ is negative, he used the following invariants: the distance $d$ between the unique quaternionic line $L(p_1, p_2)$ passing through the points $p_1, p_2$
and the point $p_3$, and the Cartan invariant $\A(p)$ of the triple $p$. Then he claimed that the  congruence class  relative to the diagonal action of  ${\rm PU}(n,1; \Q)$ of  $p$
is completely defined by $d$ and $\A(p)$. Below we give an example which shows that this claim is not correct.

\medskip

\noindent {\bf Example I.} Let  $p=(p_1, p_2, p_3)$, where $p_1, p_2$ are isotropic and $p_3$ is negative. Let $L$ be the unique quaternionic line passing through the points $p_1, p_2$.
In our example, we consider the case when the point $p_3$ is contained in $L$. Therefore, in this case, $d=0$ and $\A(p)= \pi/2$ for all triples satisfying our condition. Now we will show
that among such triples there exist triples which are not congruent relative to the diagonal action of  ${\rm PU}(n,1; \Q)$. Since the quaternionic line $L$
is isometric to ${\rm H^{4}_{\R}}$, we have that there exists a totally geodesic submanifold $P$ of $L$ of real dimension $2$ such that $p_1, p_2$ are in the ideal boundary of $P$ and $p_3 \in P$. In fact,
$P$ is a totally geodesic submanifold of ${\rm H^{n}_{\Q}}$ isometric to ${\rm H^{1}_{\C}}$. We consider $P$ as a hyperbolic plane and $p=(p_1, p_2, p_3)$ as a triangle in $P$ with two ideal vertices $p_1, p_2$ and one proper
vertex $p_3$. Let $l$ be the unique geodesic in $P$ defined by $p_1, p_2$. Let $d_l$ be the distance between $l$ and $p_3$. It is well-known from plane hyperbolic geometry, see \cite{Ber}, that a triangle $p=(p_1, p_2, p_3)$ is defined uniquely up to the isometry of $P$ by $d_l$. Therefore, fixing $p_1, p_2$ and moving $p_3$ inside of $P$, we get an infinite family of non-congruent triangles relative to the diagonal action of  ${\rm PU}(n,1; \Q)$ with the same invariants
defined in \cite{Cao}.

\subsubsection{Triangles with one ideal vertex}

Let $p=(p_1, p_2, p_3)$ be a triple of points in $\P\Q^{n}$. As usual, we
will denote by $v$ the triple $v=(v_1,v_2, v_3)$, where $v_i$ is a lift of $p_i$ in $\Q^{n,1}$.

First, we consider a configuration when  $p_1$ is isotropic and  $p_2$, $p_3$ are negative. So, in this case $p_1$ and $p_2$ and $p_3$ represent a triangle in ${\rm H^{n}_{\Q}}$
with one ideal vertex.

\medskip

Let $v_2, v_3$ be negative vectors representing $p_2, p_3$ and $v_1$ be an isotropic vector representing $p_1$.

\medskip

Consider the following quaternion:

$$\eta(v_1,v_2, v_3)= \herm{v_1,v_3}\herm{v_3,v_2}\herm{v_1,v_2}^{-1}\herm{v_3,v_3}^{-1}.$$

It is easy to check that

$$\eta(v_1\lambda_1, v_2\lambda_2, v_3\lambda_3)=$$

$$\frac{\bar{\lambda}_1}{|\lambda_1|}\eta(v_1,v_2, v_3)\frac{\lambda_1}{|\lambda_1|}.$$

\medskip

Therefore, this implies that $\eta(v_1,v_2, v_3)$ is independent of the choices of lifts of $p_2$ and $p_3$, and if we change a lift of $p_1$, we get a similar quaternion.

\medskip

By applying Proposition \ref{compt}, we can assume that $p_1$, $p_2$, $p_3$ lie in a projective submanifold $W \subset \P\Q^{n}$ of complex type of complex dimension 2 passing through the points $p_i$.
Moreover, this submanifold $W$ can be chosen, up to the action of ${\rm PU}(n,1; \Q)$, to be the canonical complex submanifold $ \P\C^{2} \subset \P\Q^{n}$. Therefore, we can assume without loss of
generality that the coordinates of the vectors  $v_1$, $v_2$,  $v_3$ are complex numbers.

\medskip

Let $G=\ (g_{ij}) = (\herm{v_i,v_j})$ be the Gram matrix associated to the points $p_1$, $p_2$, $p_3$. Then $g_{11}=0$, $g_{22} <0$,
and $g_{33}<0$. It is not difficult to show that by appropriate re-scaling we may assume that $g_{11}=0$, $g_{22}=-1$,  $g_{33}=-1$, $g_{12}=1$, $g_{23}=r_{23}> 0$, and $g_{13}$ is an arbitrary complex number.

\medskip

We call such a matrix $G$ a {\sl complex normal form} of Gram matrix associated to $p_1$, $p_2$, $p_3$ . Also, we call $G$ the {\sl complex normalized} Gram matrix.

\medskip

It is clear that  for $p_1$, $p_2$, $p_3$ any vectors $v_1$, $v_2$ and $v_3$ which represent  $p_1$, $p_2$, and $p_3$ generate a regular space in $\Q^{n,1}.$
Therefore, repeating again almost word for word the proof of Proposition \ref{pos1}, we get

\begin{prop} Let $p$ and  $p$ as above. Let  $v$ and $v^\prime$ be their lifts such that the Gram matrices $G=(g_{ij})= (\herm{v_i,v_j})$ and $G^\prime=(g^\prime_{ij})= (\herm{v^\prime_i,v^\prime_j})$ associated to $p$ and $p^\prime$
are complex normalized. Then $p$ and $p^\prime$ are equivalent relative to the diagonal action of ${\rm PU}(n,1; \Q)$ if and only if $G=G^\prime$ or $\overline{G}=G^\prime$.
\end{prop}

Again this justifies the following definition.

\medskip

Let $p=(p_1, p_2, p_3)$ be a triple of points as above.
Then the {\sl quaternionic $\eta$-invariant},  $\eta =\eta (p)$, is defined to be the unique complex number $b=b_0 +b_1$ with $b_1 \geq 0$,
in the similarity class of $\eta(v_1,v_2, v_3)$.

\medskip

\begin{theorem}\label{mix3} Let $p=(p_1, p_2, p_3)$ and  $p^\prime = (p^\prime_1, p^\prime_2, p^\prime_3)$ as above. Then $p$ and $p^\prime$ are equivalent relative to the diagonal action of  ${\rm PU}(n,1; \Q)$
if and only if $\eta (p) =\eta (p^\prime)$ and $d(p_2, p_3)= d(p^\prime_2, p^\prime_3$).
\end{theorem}

\begin{proof} It follows from the above that we can choose lifts $v=(v_1, v_2, v_3)$ and $v^\prime = (v^\prime_1, v^\prime_2, v^\prime_3)$ of
$p=(p_1, p_2, p_3)$ and $p=(p^\prime_1, p^\prime_2, p^\prime_3)$ such that the Gram matrices $G$ and $G^\prime$ associated to $p$ and $p\prime$ defined by $v$ and $v^\prime$
are complex normalized. Then it follows from the definition of the Gram matrix that $\eta (p) = g_{13}$ and $r_{23}= \sqrt{d(p_2, p_3)}.$  This implies that $G =G^\prime$.
Now the proof follows from  Lemma \ref{gram1}.
\end{proof}

\medskip

It follows from this theorem that the congruence class relative to the diagonal action of  ${\rm PU}(n,1; \Q)$  of $p=(p_1, p_2, p_3)$ is defined $\eta (p)$ and the distance between $p_2$ and $p_3$.

\medskip

The case when $p_2$, $p_3$ are positive  is similar. The $\eta$-invariant is defined by the same formula and  the proof of the congruence theorem is a slight modification of Theorem \ref{mix3}. We only remark that geometrically this configuration is equivalent to one isotropic point and two totally geodesic quaternionic hyperplanes in $\rm H^{n}_{\Q}$.

\medskip

We remark that the case when $p_1$ is isotropic and $p_2, p_3$ are negative was considered by Cao \cite{Cao}, see Theorem 1.1 item (iii). In order to describe the congruence class  relative to the diagonal action of  ${\rm PU}(n,1; \Q)$ of a triple  $p=(p_1, p_2, p_3)$, where $p_1$ is isotropic and $p_2,p_3$ are negative, he used the following invariants: the distance $d_1$ between the unique quaternionic line $L(p_1, p_2)$ passing through
the points $p_1, p_2$ and  $p_3$, the distance $d_2$  between the unique quaternionic line $L(p_1, p_3)$ passing through the points $p_1, p_3$ and  $p_2$, and the distance $d_3$ between the points $p_2, p_3$.  Then he claimed that the  congruence class  relative to the diagonal action of  ${\rm PU}(n,1; \Q)$ of a triple  $p$ is completely defined by $d_1, d_2$ and $d_3$. Below we give an example which shows that this claim is not correct.

\medskip

\noindent {\bf Example II.} This example is similar to Example I above. Let  $p=(p_1, p_2, p_3)$, where $p_1$ is isotropic and $p_2, p_3$ are negative.
In our example, we consider the case when the points $p_1, p_2, p_3$ are contained in a quaternionic line $L$. Therefore, in this case, $d_1=0$ and $d_2=0$ for all triples satisfying our condition. Now we will show
that among such triples there exist triples which are not congruent relative to the diagonal action of  ${\rm PU}(n,1; \Q)$. Since the quaternionic line $L$
is isometric to ${\rm H^{4}_{\R}}$, we have that there exists a totally geodesic submanifold $P$ of $L$ of real dimension $2$ such that $p_1$ is in the ideal boundary of $P$, and $p_2, p_3 \in P$. As in the example above, we have that
$P$ is a totally geodesic submanifold of ${\rm H^{n}_{\Q}}$ isometric to ${\rm H^{1}_{\C}}$. We consider $P$ as a hyperbolic plane and $p=(p_1, p_2, p_3)$ as a triangle in $P$ with one ideal vertex $p_1$ and two proper vertices $p_2, p_3$.
It is well-known from plane hyperbolic geometry, see \cite{Ber}, that a triangle $p=(p_1, p_2, p_3)$ is defined uniquely up to the isometry
of $P$ by $d_3$ and by its angles at the vertices $p_2, p_3$.  Therefore, fixing $p_2, p_3$ and moving $p_1$ along the ideal boundary of $P$, we get an infinite family of triangles with the same $d_3$ and with different angles at the vertices $p_2, p_3$. This implies that there exist an infinite family of  non-congruent triangles  relative to the diagonal action of  ${\rm PU}(n,1; \Q)$ with the same invariants defined in \cite{Cao}.

\subsubsection{Triangles with one negative vertex and two positive vertices}

Let $p=(p_1, p_2, p_3)$ be a triple of points in $\P\Q^{n}$. In this section, we consider a configuration when  $p_1$ is negative, and  $p_2$, $p_3$ are positive. So, in this case $p_1$ represents a point in ${\rm H^{n}_{\Q}},$
and $p_2$ and $p_3$ represent two totally geodesic quaternionic hyperplane $H_2$ and $H_3$ in $\rm H^{n}_{\Q}$.

\medskip
Let $v_1$ a negative  vector representing $p_1$ and $v_2$, $v_3$ be positive vectors representing $p_2, p_3$.

\medskip

Let $\eta(v_1,v_2, v_3)$ be following quaternion:
$$\eta(v_1,v_2, v_3)= \herm{v_1,v_3}\herm{v_3,v_2}\herm{v_1,v_2}^{-1}{ \herm{v_3,v_3}}^{-1}.$$

\medskip

We see that $\eta(v_1,v_2, v_3)$ is not well-defined when the points $p_1$ and $p_2$ are orthogonal, that is,  $\herm{v_1,v_2}=0$. It follow from the definition of polar points
that $p_1$ is orthogonal to  $p_2$ if and only if $p_1 \in H_2$. Also we see that $\eta(v_1,v_2, v_3)=0$ when $p_1$ is orthogonal to $p_3$ or $p_2$ is orthogonal to $p_3$.
We analyze  all these special configurations in the end of this section and show that in all these cases we need less invariants to describe the congruence class then in generic case.

\medskip

We say that a triple $p=(p_1, p_2, p_3)$ as above is {\sl generic} if and only if  all the pairs $(p_i, p_j)$ are not orthogonal, $i\neq j$.

\medskip

In what follows, let $p=(p_1, p_2, p_3)$ be a generic triple as above. It is easy to check that
$$\eta(v_1\lambda_1, v_2\lambda_2, v_3\lambda_3)= \frac{\bar{\lambda}_1}{|\lambda_1|}\eta(v_1,v_2, v_3)\frac{\lambda_1}{|\lambda_1|}.$$
Therefore, this implies that $\eta(v_1,v_2, v_3)$ is independent of the choices of lifts of $p_2$ and $p_3$, and if we change a lift of $p_1$, we get a similar quaternion.

\medskip

By applying Proposition \ref{compt} again, we can assume that $p_1$, $p_2$, $p_3$ lie in a projective submanifold $W \subset \P\Q^{n}$ of complex type of complex dimension $2$.
Moreover, this submanifold $W$ can be chosen, up to the action of ${\rm PU}(n,1; \Q)$, to be the canonical complex submanifold $ \P\C^{2} \subset \P\Q^{n}$.
Therefore, we can assume without loss of generality that the coordinates of the vectors  $v_1$, $v_2$,  $v_3$ are complex numbers.

\medskip

Let $G=\ (g_{ij}) = (\herm{v_i,v_j})$ be the Gram matrix associated to the triple of points $(p_1$,\  $p_2$,\  $p_3$). Then $g_{11}<0$,\  $g_{22} >0$,
and $g_{33}>0$. It is not difficult to show that by appropriate re-scaling we may assume that $g_{11}=-1$,\  $g_{22}=1$,\   $g_{33}=1$, $g_{12}=r_{12}>0$,\  $g_{13}=r_{13}> 0$,
and $g_{23}$ is an arbitrary complex number.

\medskip

We call such a matrix $G$ a {\sl complex normal form} of Gram matrix associated to $(p_1$,\  $p_2$,\ $p_3)$. Also, we call $G$ the {\sl complex normalized} Gram matrix.

\medskip

As before we have

\begin{prop} Let $p$ and  $p^\prime)$ as above. Let  $v$ and $v^\prime$ be their lifts such that the Gram matrices $G=(g_{ij})= (\herm{v_i,v_j})$
and $G^\prime=(g^\prime_{ij})= (\herm{v^\prime _i,v^\prime _j})$ associated to $p$ and $p^\prime$
are complex normalized. Then $p$ and $p^\prime$ are equivalent relative to the diagonal action of ${\rm PU}(n,1; \Q)$ if and only if $G=G^\prime$ or $\overline{G}=G^\prime$.
\end{prop}

\medskip

Again this justifies the following definition.

\medskip

Let $p=(p_1, p_2, p_3)$ be a triple of points as above.
Then the {\sl quaternionic $\eta$-invariant},  $\eta =\eta (p)$, is defined to be the unique complex number $b=b_0 +b_1$ with $b_1 \geq 0$ in the similarity class of $\eta(v_1,v_2, v_3)$.

\medskip

Now we define one more invariant. Let $p$ be a negative point and $q$ be a positive point.
Let $v$ be a vector representing $p$ and $w$ be a vector representing $q$. Then we define

$$d(p,q)=d(v,w) = \frac{\herm{v,w}\herm{w,v}}{\herm{v,v}\herm{w,w}}.$$

It is easy to see that $d(p,q)$ is independent of the chosen
lifts  $v, w$, and that $d(p,q)$ is invariant with respect to the diagonal action of ${\rm PU}(n,1; \Q)$.

\medskip
We call $d(p,q)$ the {\sl mixed distant invariant} associated to the points $p,q$.

\medskip

Below we will show that this invariant defines the
distance $\rho$ from the point $p \in {\rm H^{n}_{\Q}}$ to the totally geodesic quaternionic hyperplane $H$ in $\rm H^{n}_{\Q}$ whose polar point is $q$.

\medskip

\begin{prop} Let $p$ be a negative point and $q$ be a positive point. Let $v$ be a vector representing $p$ and $w$ be a vector representing $q$. Then the distance $\rho$ from the point $p \in {\rm H^{n}_{\Q}}$ to the totally geodesic quaternionic hyperplane $H$ in $\rm H^{n}_{\Q}$ with polar point $q$ is defined by the following equality
$$\sinh ^2(\rho (p, H))= -d(p,q).$$
\end{prop}
\begin{proof} To prove this, we use the ball model:

$$\rm H^{n}_{\Q} = \B = \{z=(z_1, \ldots, z_n) \in \Q^n : \Sigma_{i=1}^n |z_i|^2 < 1\},$$
see Section 1.2.2.

\medskip

In what follows, for any $z \in \B$, we write $Z=(z_1, \ldots, z_n, 1)$ for the standard
lift of $z$ in $ \Q^{n,1}.$

\medskip

Let $H$ be the totally geodesic quaternionic hyperplane in $\rm H^{n}_{\Q}$ with polar point $q$. It is easy to see that the distance from $p$ to
$H$ is equal to the distance between $p$ and its orthogonal projection $p^*$ onto $H$. Applying an automorphism, we may assume that 

$$H= \{z=(0,z_2 \ldots, z_n) \subset \B.$$

Therefore, if $p=(z_1, \ldots, z_n)$, then $p^* = (0,z_2 \ldots, z_n).$

\medskip

So, we have that $d(p,H)=d(p,p^*)$. Then applying the distance formula in  $\rm H^{n}_{\Q}$, 
we get

$$\cosh^2 (d(p,p^*))=  \frac{\herm{Z,Z'}\herm{Z',Z}}{\herm{Z,Z}\herm{Z',Z'}} =$$

$$\frac{(\vert z_2 \vert^2 +\ldots +\vert z_n \vert^2 -1)^2}{(\vert z_1 \vert^2 +\vert z_2 \vert^2 +\ldots + \vert z_n \vert^2 -1)(\vert z_2 \vert^2 +\ldots +\vert z_n \vert^2 -1)}=$$

$$\frac{\vert z_2 \vert^2 +\ldots +\vert z_n \vert^2 -1}{\vert z_1 \vert^2 +\vert z_2 \vert^2 +\ldots + \vert z_n \vert^2 -1}=$$

$$1- \frac{\vert z_1 \vert^2 }{\vert z_1 \vert^2 +\vert z_2 \vert^2 +\ldots + \vert z_n \vert^2 -1.}$$

Then, by taking $w=(1, 0, \ldots, 0)$ as a polar vector to $H$, we get that 

$$\sinh ^2(\rho (p, H))= -d(p,q).$$

\end{proof}

\begin{theorem}\label{243t} Let $p=(p_1, p_2, p_3)$ and  $p^\prime = (p^\prime_1, p^\prime_2, p^\prime_3)$ as above. Then $p$ and $p^\prime$ are equivalent relative to the diagonal action of  ${\rm PU}(n,1; \Q)$
if and only if $\eta (p) =\eta (p^\prime)$,\  $ d(p_1, p_2)= d(p^\prime_1,\  p^\prime_2)$, and  $d(p_1, p_3)= d(p^\prime_1,\  p^\prime_3)$.
\end{theorem}

\begin{proof} It follows from the above that we can choose lifts $v=(v_1,v_2, v_3)$ and $v^\prime = (v^\prime_1,v^\prime_2,v^\prime_3)$ of
$p=(p_1,p_2,p_3)$ and $p=(p^\prime_1, p^\prime_2, p^\prime_3)$ such that the Gram matrices $G$ and $G^\prime$ associated to $p$ and $p^\prime$ defined by $v$ and $v^\prime$
are complex normalized. Then it follows from the definition of the Gram matrix that $r_{12}= \sqrt{d(p_1, p_2)},$ $r_{13}= \sqrt{d(p_1, p_3)},$ and  $g_{23} = \overline{\eta (p)} (r_{12}/r_{13})$.
This implies that $G =G^\prime$.  Now the proof follows from Lemma \ref{gram3}.
\end{proof}

\medskip

Now we consider some special configurations of triples $p=(p_1, p_2, p_3).$ Let, for instance, $p$ be a configuration where all the pairs $(p_i,p_j)$, $i\neq j$, are orthogonal. In this case,
the totally geodesic quaternionic hyperplanes $H_1$ and $H_2$ with polar points $p_2$ and $p_3$ intersect orthogonally and $p_1$ lies in their intersection. It is clear then that all such configurations are
congruent relative to the diagonal action of ${\rm PU}(n,1; \Q)$. So, the moduli space in this case is trivial. Another example: suppose that $p_1$ is orthogonal to $p_2$, and
$p_1$ is orthogonal to $p_3$. This implies that $H_1$ and $H_2$ intersect, and $p_1$ lies in their intersection. This implies that in order to describe the congruence class of such configuration we need only the d-invariant
of $H_1$ and $H_2$, the angle between $H_1$ and $H_2$. Another special configuration may be analyzed easily  in a similar way.

\subsubsection{Triangles with one positive vertex and two negative vertices}

In this section, we consider a triple $(p_1, p_2, p_3)$, where $p_1$ is positive and $p_2$, $p_3$ are negative.

\medskip

Geometrically, this corresponds to a totally geodesic quaternionic hyperplane $H$ in $\rm H^{n}_{\Q}$ and two points $p_2$ and $p_3$ in $\rm H^{n}_{\Q}$.

\begin{theorem} Let $H$ be a totally geodesic quaternionic hyperplane $H$ in $\rm H^{n}_{\Q}$ with polar point $p_1$,  and $p_2$ and $p_3$ in $\rm H^{n}_{\Q}$.
Suppose that $p=(p_1, p_2,p_3)$ is generic. Then the congruence class of the triple $(H, p_1, p_2)$ is defined uniquely by two distances  $\rho (p_1, H)$, $\rho (p_1, H)$, and
$\eta (p)=\eta (p_1, p_2,p_3)$.
\end{theorem}

\begin{proof} The proof of this theorem is a slight modification of the proof of Theorem \ref{243t}.
\end{proof}

\section*{Declarations}
%\textbf{Data availability} Not applicable.\\
%\\
\textbf{Conflict of interest} The authors have no conflict of interest to declare that are relevant to this article.

%--------------------------------------------------------------------
\end{document}